\documentclass[11pt]{article}

\usepackage{amsmath}
\usepackage{amssymb}
\usepackage{amsthm}
\usepackage{graphicx}
\usepackage{caption}
\usepackage{fourier}
\usepackage{soul} 

\usepackage{color}
\usepackage{comment}

\usepackage[margin=3cm]{geometry}

\usepackage[T1]{fontenc}

\usepackage{nicefrac}

\usepackage{cancel} 
\usepackage{hyperref}

\usepackage[textwidth=2.5cm]{todonotes}
\presetkeys{todonotes}{size=\scriptsize}{}

\newtheorem{theorem}{Theorem}
\newtheorem{proposition}{Proposition}
\newtheorem{lemma}{Lemma}

\theoremstyle{definition}

\theoremstyle{remark}
\newtheorem{remark}{Remark}
\newtheorem{example}{Example}

\def\NN{\mathbb{N}}

\def\11{\mathds{1}}

\def\E{\mathbb{E}}
\def\P{\mathbb{P}}
\def\R{\mathbb{R}}

\def\1{\mathbf{1}}
\def\N{\mathbb{N}}

\def\cF{{\cal F}}

\def\cX{{\cal X}}

\def\loga{\log^{(a)}}
\def\logone{\log^{(1)}}
\def\logalpha{\log^{(\alpha)}}
\def\expone{\exp^{(1)}}

\newcommand{\XX}{{\mathbf X}}

\begin{document}
    
    \title{Convergence of weighted branching processes}
\author{Denis Villemonais\footnote{IRMA, Université de Strasbourg, France}$\ \,$\footnote{Institut Universitaire de France}$\ $ and Nicol\'as Zalduendo\footnote{Centro de Modelamiento Matem\'atico (CNRS IRL2807), Universidad
		de Chile, Santiago, Chile\\ \texttt{denis.villemonais@unistra.fr, nzalduendo@dim.uchile.cl}}}

    \maketitle

    \begin{abstract}
        We study the long-term behavior of weighted multi-type branching processes, focusing on  extending classical laws of large numbers and martingale convergence to settings with infinitely many weighted particles, arbitrary type spaces and non-geometric rescaling. We demonstrate applications to Galton-Watson trees indexed by random weights and by random kernels, convergence in Wasserstein distance of the underlying mean semi-group,  and convergence of ergodic averages along lineages.
    \end{abstract}

 

    \section{Introduction}

    We consider discrete time branching type models with weights, allowing infinitely many particles per generation but finite total mass.
    Informally,  we have in mind a branching process with type space $\XX$ where any particle of type $x\in \XX$ and weight $w\in[0,+\infty)$ gives birth at the next generation to a family of weighted particles $Y_{i}$ with respective weights $ww_i$, $i\geq 1$, where the family $(w_i,Y_i)_{i\geq 1}$ is independent from the rest of the branching process and has a law which only depends on the type $x$ of the parent particle. Iterating this process allows to define a sequence of families, indexed by time and which we will refer to as the generations, of weighted and typed particles.
    In this article, we focus on conditions ensuring that, up to a multiplicative factor, the corresponding sequence of weighted empirical measures converges when the time goes to infinity.

    Branching models with weights but not types have received much
    attention and are usually referred to as Mandelbrot's martingale~\cite{liang2015weighted}, Mandelbrot's
    cascade~\cite{liu2000generalized} or simply weighted
    branching processes~\cite{vatutin2002rate,roesler2016weighted}. In
    particular, the limiting behavior is well
    known and necessary and sufficient conditions for the
    existence of an exponential moment for its limit have been obtained~\cite{bingham1975asymptotic} when weights are bounded by $1$ and extensively studied in subsequent
    works in the unbounded-weights case~\cite{kahane1976certaines,durrett1983fixed,liu1996products,liu2000generalized,alsmeyer2005double,liang2015weighted}. 
    We emphasize that the classical well-studied branching model with types corresponds to the situation where all the weights equal $1$, with finitely many particles in each generation, see e.g.~\cite{Harris1963,AsmussenHering1976} and their generalized version~\cite{Jagers1989}.

    In order to define the process, we assume that we are given a measurable type space $(\XX,\mathcal X)$ and 
    a probability kernel $K$ from $\XX$ to $(\mathbb R_+\times \XX)^{\mathbb N}$, $\mathbb N=\{0,1,\dots\}$, endowed with its cylindrical $\sigma$-field. Informally, $K$ determines the law of individuals progeny: given a parent with weight $u\in \mathbb R^+$ and type $x\in \XX$, its progeny is given by individuals with weights $uu_i$ and types $x_i$, $i\in \mathbb N$, where $(u_i,x_i)_{i\in \mathbb N}$ has law $K(x,\cdot)$.

    Each generation is represented by a weighted discrete measure on $\XX$, and  the initial population is denoted by
    \[
    G_0=\sum_{e\in\mathbb N} w_e\delta_{X_e},\quad\text{ where }(w_e,X_e)_{e\in\mathbb N}\in (\mathbb R_+\times \XX)^{\mathbb N},
    \]
    where $ei$ denotes the concatenation of $e$ and $i$. We define iteratively the population at each generation $n+1$, $n\geq 0$, denoted by 
    \[
    G_{n+1}=\sum_{e\in\mathbb N^{n+1}}  w_{e} \delta_{X_{e}}=\sum_{e\in\mathbb N^{n},i\in\mathbb N}  w_{e}u_{ei} \delta_{X_{ei}},
    \]
    where the $(u_{ei},X_{ei})_{i\in\mathbb N}$ are distributed according to $K(X_e,\cdot)$ and  independent from each  other conditionnaly to 
    \[
    \mathcal F_n:=\sigma(w_{e'},X_{e'},\ e'\in\mathbb N^k,\ k\leq n).
    \] 
    These models can be seen as a generalization of
    weighted multi-type Galton Watson processes with arbitrary type space
    and weights, where, informally, $(w_{ei}X_{ei})_{i\in\mathbb N}$ is seen as the weighted progeny of the individual with label $e$ and type $X_e$.
    
    We define the  non-negative kernel $(Q_x)_{x\in \XX}$
    from $\XX$ to itself as follows: for all $n\geq 0$, $e\in\mathbb N^n$, $x\in \XX$ and $A\in\cX$,
    \begin{align}
    Q_x(A):=\mathbb E\left(\sum_{i\in \mathbb N} u_i\mathbf 1_{Y_i\in A}\right),\ \text{ where }(u_i,Y_i)_{i\in\mathbb N}\text{ is distributed according to }K(x,\cdot).
    \end{align}
    In particular, for all $n\geq 0$, $e\in\mathbb N^n$, $x\in \XX$ and $A\in\cX$,
    \begin{align}
        \label{eq:def-Qx}
       \E(G_{n+1}(A)\mid \mathcal F_n)= \E\left(\sum_{i\in\mathbb N}  w_{ei} \mathbf 1_{X_{ei}\in A}\mid\cF_{n}\right)= w_e\,Q_{X_e}(A).
    \end{align}
    One easily observes that, for all $m,n\geq 0$,
    \begin{align}
        \label{eq:intro-many-to-one}
        \mathbb E(G_{m+n}\mid \mathcal F_m)=G_m Q^{n},
    \end{align}
    and one thus expect that the asymptotic behaviour of $G_n$ will be largely determined by the asymptotic behaviour of the iterated kernels $Q^n$ when $n\to+\infty$.
    The main aim of this paper is to study the long time behaviour of
    $(G_n)_{n\geq 0}$. 
    
    For classical branching process in discrete time, this behaviour is well understood: Asmussen and Hering~\cite{AsmussenHering1976} have proposed a robust proof technique that establishes similar laws of large numbers under strong convergence assumptions in total variation. In the different context of Mandelbrot cascades, significant similar results have long been known, see e.g.~\cite{bingham1975asymptotic} The first objective of this paper is to extend these results within a unified framework—namely, that of weighted multi-type branching processes. Furthermore, we consider convergence under more general conditions, and allow for instance convergence in Wasserstein distance or for only specific test functions. In addition,  we allow the spectral vector spaces of the expectation semigroup to be multi-dimensional, which is particularly relevant for reducible processes. We also extend the classical geometric rescaling of the semigroup to a combined polynomial-geometric scaling, which naturally appears for non-irreducible semi-groups. These generalizations broaden the scope and applicability of classical results in the theory of branching processes. 
    
    Let us also mention, as was already observed in~\cite{AsmussenHering1976} in the classical case, that continuous time branching processes discretized in time are discrete time branching processes, so that results can be transferred from the latter setting to the former, and in particular the $L^p$ convergence can be extended directly. There are however several technical difficulties to extend the almost sure convergence, some of which can be overcome using the method of~\cite{AsmussenHering1976} or by using model specific technics (see e.g.~\cite{HortonKyprianou2023} and references therein). For weighted branching processes, there is however an additional difficulty: since infinitely many particles are allowed, there are actually infinitely many branching events in any open interval of time. This difficulty requires a different construction and approach compared to the usual framework and is the subject of an ongoing work.

     In Section~\ref{secDescription}, we state and prove sufficient conditions for the $L^p$ convergence of properly rescaled weighted branching processes. In Section~\ref{secXlogX}, we prove a $L\log L$-type criterion for the branching process. Finally, in Section~\ref{sec:examples}, we apply our results to four different models: we consider products of random weights indexed by Galton-Watson trees (Section~\ref{sec:exaGW}), products of random kernels indexed by Galton-Watson trees (Section~\ref{sec:exakernels}), weighted branching processes whose mean semi-kernel contracts in Wasserstein distance (Section~\ref{sec:exaWasserstein}) and convergence of ergodic averages along lineages (Section~\ref{sec:exainhomogeneous}).

    \section{Law of large numbers in $L^p$}
    \label{secDescription}
    
    The aim of this section is to provide a strong law of large numbers in
    $L^p$ for the sequence $(G_n)_{n\in\N}$.
    Our main assumption is made of a strong mixing condition on
    the kernel $Q$ and a $L_p$ bound on the dispersion of the random
    progeny.

For any $p\in(1,2]$, we define the kernel $Q^{(p)}$ from $\XX$ to itself defined as follows: for all $x\in \XX$ and $A\in\mathcal X$,
\begin{align*}
       Q^{(p)}_{x} A:= \mathbb E\left(\sum_{i\geq 1}u_i^p\,\mathbf 1_{Y_i\in A}\right),\quad\text{ where }(u_i,Y_i)_{i\geq 1}\text{ is distributed according to $K(x,\cdot)$}.
\end{align*}
Defining
\begin{align*}
    G^{(p)}_n=\sum_{e\in \mathbb N^n}w_e^p \delta_{X_e},
\end{align*}
we observe that
\begin{align*}
    \mathbb E(G^{(p)}_n)= \mathbb E(G^{(p)}_0)(Q^{(p)})^n
\end{align*}

In the following assumption and in the rest of the paper, the notation $\|\cdot\|_{\psi_1}$ stands for the weighted infinite norm $\|g\|_{\psi_1}=\sup_{x\in X}|g(x)/\psi_1(x)|$, for all measurable function $g:X\to\mathbb R$.

    \medskip \noindent\textbf{Assumption MD$(f,p)$.} We have   $p\in(1,2]$ and $f:\XX\to \mathbb R$ is a measurable function. In addition, there exists two measurable functions $\psi_1:\XX\to(0,+\infty)$ and $\psi_2:\XX\to[0,+\infty)$, and a constant $\theta_1>0$ such that
    \begin{enumerate}
        \item $\|f\|_{\psi_1}<\infty$ and there exists a sequence of non-negative
        numbers $(\alpha_n)_{n\in\N}$ which converges to $0$, a non-negative integer $\beta$, a function $\eta_f:\XX\to \mathbb R$  and a constant  $c_1>0$ such that
        \begin{align}
            \label{eq:eqMD1}
            \left|n^{-\beta}\theta_1^{-n}\delta_x Q^n f-\eta_f(x)\right|\leq \alpha_n  \,\psi_1(x)      ,\quad \forall n\geq 1,\ x\in \XX;
        \end{align}
        and
        \begin{align}
            \label{eq:eqMD2}
            n^{-\beta}\theta_1^{-n}\delta_x Q^n \psi_1\leq c_1\,\psi_1(x)
            ,\quad\forall n\geq 0,\,x\in \XX;
        \end{align}

        \item there exist two constants $c_2,c_3>0$, a summable non-negative sequence $(\gamma_n)_{n\in\mathbb N}$ and a measurable function $\psi_2$ such that 
            \begin{align}
            	\label{eq:eqMD3bis1}
            \theta_1^{-pn}\delta_x (Q^{(p)})^n \psi_2^p\leq c_2\,\gamma_n^p \psi_2^p(x)\,/\,n^{(p/2-1)_+},
            \end{align}
             and, for all measurable function $g:\XX\to \mathbb R$ and all $x\in \XX$,
            \begin{align}
                \label{eq:eqMD3bis}
                 \E\left(\left|\sum_{i\in\mathbb N} u_i g(Y_i)-\delta_x Q g\right|^p\right)\leq c_3\psi_2(x)^{p} \|g\|_{\psi_1}^p,
            \end{align}
            where $(u_i,Y_i)_{i\in\mathbb N}$ is distributed according to $K(x,\cdot)$.
    \end{enumerate}

\begin{remark}
    \begin{itemize}
        \item  In the literature, properties similar to 
the first part of Assumption~MD($f,p$) are obtained for several functions $f$ at once, using classical strong mixing criteria in total variation spaces, see e.g. the reference books~\cite{MeynTweedie2009,DoucMoulinesEtAl2018} for conservative semi-groups and~\cite{DelMoral2004,DelMoral2013} for non-conservative semi-groups. While most of related works for branching processes focus on exponential renormalisation (i.e. $\beta=0$), one novelty of our work is to allow for polynomial renormalisation (i.e. $\beta>0$), which is typically needed for non-conservative processes in reducible state spaces (see e.g.~\cite{ChampagnatVillemonais2022,villemonais2025quasi} and references therein).
    \item Another particularity of Assumption MD($f,p$) is that we allow  $\beta$, $\theta_1$, $\nu_Q$ and $\eta$ to depend on $f$. This is particularly useful when the optimal coefficients $\beta$ and $\theta_1$ depend on $f$ (this is typically the case in reducible state spaces) and when the convergence only holds true for a particular subset of functions (e.g. regular ones, which is typically the case when the convergence holds true in Wasserstein distance).
    \item We also emphasize that, similarly as in~\cite{Jagers1989}, we consider the $L^p$ dispersion, instead of the more usual $L^p$ moment, of $\sum_{i\in\mathbb N} u_i g(Y_i)$. This allows for large values/low variation of the progeny.
    \end{itemize}
\end{remark}

    \begin{theorem}
        \label{thmMain}
        Assume that Assumptions~MD($f,p$) holds true for some function $f$ and $p\in(1,2]$, and set $\Gamma_m=\sum_{k\geq m}\gamma_k$. Then there exists a constant $C>0$ such that
    \begin{multline}
    \label{eq:eqthmMain}
        \sqrt[p]{ \E\left(\left|(n+m)^{-\beta}\theta_1^{-(m+n)}G_{m+n}(f)-W^f_\infty\right|^p\right)}\\ 
        \begin{aligned}
            \leq&\quad c_0\theta_1^{-1} \left(c_1\|f\|_{\psi_1}+\|\eta_f\|_{\psi_1}\right)\,\Gamma_m  \sqrt[p]{\E \left(G_0^{(p)} (\psi_2^p)\right)}\\
            &+\left(\alpha_n \dfrac{n^\beta m^\beta}{(n+m)^\beta}c_1 + \|\eta_f\|_{\psi_1} \left(1-\dfrac{n^\beta}{(n+m)^\beta}\right)\right) \left(c_0 \theta_1^{-1}\,\Gamma_0 \sqrt[p]{\E \left(G_0^{(p)} (\psi_2^p)\right)} + \sqrt[p]{\E\left(\left|G_0\psi_1\right|^p\right)}\right).
        \end{aligned}
    \end{multline}
        where
        $W_\infty^f$ is the almost sure and $L^p$ limit of
        the uniformly
        integrable martingale $W^f:=(\theta_1^{-n}G_n(\eta_f))_{n\in\N}$, and with $c_0=\sqrt[p]{2c_2c_3}$.
    \end{theorem}
    
    In the proof of Theorem~1, we will make use of the second part of Assumption~MD($f,p$) through the following lemma.
    \begin{lemma}
        \label{prop:suffcondMD2}
       Let $g:\XX\to\mathbb R$ such that $\|g\|_{\psi_1}<\infty$. Then, for all $n\geq 0$, 
        \begin{align}
            \label{eq:eqMD3}
            \sqrt[p]{\E\left(\left| G_{n+1}g-G_nQg\right|^p\right)}\leq c_0 \gamma_n \theta_1^{n} (\mathbb E(G_0^{(p)}\psi_2(x_0)^p)^{1/p} \|g\|_{\psi_1},
        \end{align}
         with $c_0=\sqrt[p]{C_pc_2c_3}$, where 
        \[
        C_p=\begin{cases}
            2&\text{ if }p\in(1,2]\\
            [8(p-1)\max(1,2^{p-3})]^p&\text{ if } p>2.
        \end{cases}
        \]
    \end{lemma}
    
      \begin{proof}[Proof of Lemma~\ref{prop:suffcondMD2}]
        Assume without loss of generality that $G_0$ is deterministic. We have
        \begin{align*}
            G_{n+1}g-G_nQg=\sum_{e\in \mathbb N^n} w_e \left(\sum_{i\in \mathbb N} u_{ei} g(X_{ei})-Q_{X_e} g\right),
        \end{align*}
        where the second term is a sum of martingale increments. 
        In particular, for all $p\in(1,2]$, according e.g. to~\cite{Chatterji1969},
        \begin{align*}
            \mathbb E\left[\left|G_{n+1}g-G_nQg\right|^p\right]
            &\leq 2\sum_{e\in \mathbb N^n} \mathbb E\left[w_e^p\left|\sum_{i\in \mathbb N} u_{ei} g(X_{ei})-Q_{X_e} g\right|^p\right]\\
            &\leq 2 c_3\|g\|_{\psi_1}^p\sum_{e\in \mathbb N^n} \mathbb E\left[w_e^p\psi_2(X_e)^p\right] \qquad\text{ using~\eqref{eq:eqMD3bis}}\\
            &= 2 c_3\,\|g\|_{\psi_1}^p \mathbb E\left[G^{(p)}_n \psi_2^p\right]\\
            &= 2 c_3\,\|g\|_{\psi_1}^p \,G^{(p)}_0 (Q^{(p)})^n \psi_2^p\qquad\text{using~\eqref{eq:intro-many-to-one}}\\
            &\leq 2 c_2\,c_3\,\|g\|_{\psi_1}^p \,\gamma_n^p\,G_0^{(p)}\psi_2^p\qquad\text{using~\eqref{eq:eqMD3bis1},}
        \end{align*}
        which concludes the proof of Lemma~\ref{prop:suffcondMD2} when $p\leq 2$ (in this case $(p/2-1)_+=0$).
    \end{proof}
    
    \begin{lemma}
        \label{lem:etaf-eigenfunction}
        We have $\|\eta_f\|_{\psi_1}<\infty$ and $Q\eta_f=\theta_1\eta_f$. 
    \end{lemma}
    
    \begin{proof}[Proof of Lemma~\ref{lem:etaf-eigenfunction}]
    The first assertion is an immediate consequence of~\eqref{eq:eqMD1} and~\eqref{eq:eqMD2} with $n=1$. Let us now prove the second one. 
            For all $x\in \XX$, we have, for all $n\geq 1$,
        \begin{align*}
            \left|n^{-\beta}\theta_1^{-n}\delta_x Q^n f-\theta_1^{-1}Q\eta_f(x)\right|
            &\leq \theta_1^{-1}\int_{\XX}\delta_xQ(\mathrm dy) \left|n^{-\beta}\theta_1^{-(n-1)}\delta_y Q^{n-1} f-\eta_f(y)\right|\\
            &\leq \theta_1^{-1}\int_{\XX}\delta_xQ(\mathrm dy) \left|(n-1)^{-\beta}\theta_1^{-(n-1)}\delta_y Q^{n-1} f-\eta_f(y)\right|\\
            &\qquad+\theta_1^{-1}\left(\left(\frac{n-1}{n}\right)^{-\beta}-1\right)\int_{\XX}\delta_xQ(\mathrm dy)\left|n^{-\beta}\theta_1^{-(n-1)}\delta_y Q^{n-1} f\right|\\
            &\leq \theta_1^{-1}\alpha_{n-1}\delta_xQ\psi_1+\theta_1^{-1}\left(\left(\frac{n-1}{n}\right)^{-\beta}-1\right)\left(\eta_f(x)+\alpha_{n-1}\psi_1(x)\right),
        \end{align*}
        where we used~\eqref{eq:eqMD1}. We deduce that $ n^{-\beta}\theta_1^{-n}\delta_x Q^n f$ converges to $\theta_1^{-1}Q\eta_f(x)$ when $n\to+\infty$, and hence, using~\eqref{eq:eqMD1} again, that $\theta_1^{-1}Q\eta_f(x)=\eta_f(x)$. This concludes the proof of Lemma~\ref{lem:etaf-eigenfunction}.
    \end{proof}
    \begin{proof}[Proof of Theorem~\ref{thmMain}]
        We assume that
        $\E\left(G_0 (\psi_2^p)+(G_0\psi_1)^p\right)<+\infty$ , the
        result being trivial otherwise.
        
        According to~\eqref{eq:intro-many-to-one} and using Lemma~\ref{lem:etaf-eigenfunction}, we deduce that, for all $n\geq 0$, 
        \begin{align*}
            \mathbb E(G_{n+1}(\eta_f)\mid \mathcal F_n)=G_n Q\eta_f=\theta_1G_n\eta_f.
        \end{align*}
        This entails that the process $W^f$ is a
        martingale. We will show that it is uniformly integrable and hence that it converges almost surely to a
        limit which we denote by $W^f_\infty$.
        
        We have
        \begin{align}
            \label{eqStep0}
            \sqrt[p]{ \E\left(\left|(m+n)^{-\beta} \theta_1^{-(m+n)}G_{m+n}(f)-W^f_\infty\right|^p\right)}\leq A_{m,n}(f)+B_{m,n}(f)+C_{m,n}(f),
        \end{align}
        where
        \begin{align*}
            &A_{m,n}(f)=\sqrt[p]{\E\left(\left|(m+n)^{-\beta}\theta_1^{-(m+n)}G_{m+n}(f)-(m+n)^{-\beta}\theta_1^{-(m+n)}G_mQ^nf\right|^p\right)}\\
            &B_{m,n}(f)=\sqrt[p]{\E\left(\left|(m+n)^{-\beta}\theta_1^{-(m+n)}G_m (Q^n f)-\frac{n^\beta}{(m+n)^\beta}\theta_1^{-m}G_m\eta_f\right|^p\right)},\\
            &C_{m,n}(f)=\sqrt[p]{\E\left(\left|\frac{n^\beta}{(m+n)^\beta}W^f_{m}-W^f_\infty\right|^p\right)}.
        \end{align*}
        The terms $A_{m,n}$, $B_{m,n}$ and $C_{m,n}$ and $D_m$ are bounded in Steps~1,~2 and~3 respectively.

        \medskip\noindent\textit{Step 1. An upper bound on $A_{m,n}$ over
            $L^\infty(\psi_1)$.}  We have, for all $h\in L^\infty(\psi_1)$,
        \begin{align*}
            (m+n)^\beta\theta_1^{m+n}\, A_{m,n}(h)=\sqrt[p]{ \E\left(\left|G_{m+n}h-G_m (Q^n h)\right|^p\right)}\leq \sum_{k=0}^{n-1} \sqrt[p]{\E\left(\left|G_{m+k+1}(Q^{n-k-1}h)-G_{m+k}(Q^{n-k})h\right|^p\right)}.
        \end{align*}
        Using Lemma~\ref{prop:suffcondMD2} with $g=Q^{n-k-1}h\in L^\infty(\psi_1)$, we deduce that
        \begin{align}\label{eq:bound_A}
            (m+n)^\beta\theta_1^{m+n}\,A_{m,n}(h)&\leq c_0\sqrt[p]{\mathbb E\left(G_0^{(p)}(\psi_2^p)\right)}\,\sum_{k=0}^{n-1} \gamma_{m+k}\theta_1^{m+k}\|Q^{n-k-1}h\|_{\psi_1}.
        \end{align}
        According to~\eqref{eq:eqMD2}, we have
        $\|Q^{n-k-1} h\|_{\psi_1}\leq
        c_1\,(n-k-1)^{\beta}\,\theta_1^{n-k-1}\|h\|_{\psi_1}$, and hence 
        \begin{align}
            (m+n)^\beta\theta_1^{m+n}\,A_{m,n}(h) &\leq \|h\|_{\psi_1}\,c_0\sqrt[p]{\mathbb E(G_0^{(p)}(\psi_2^p))}\,c_1\,\sum_{k=0}^{n-1} \gamma_{m+k}\theta_1^{m+k}\theta_1^{n-k-1}(n-k-1)^{\beta}\nonumber\\
            &\leq \|h\|_{\psi_1}\,c_0c_1\,\Gamma_m\sqrt[p]{\mathbb E(G_0^{(p)}(\psi_2^p))} \theta_1^{n+m-1}n^\beta.\label{eq:bound_Amn}
        \end{align}
        Finally, we obtain
        \begin{align}
            \label{eqStep1}
            A_{m,n}(h)
            &\leq \|h\|_{\psi_1}\,c_0c_1\theta_1^{-1} \sqrt[p]{\E(G_0^{(p)}(\psi_2^p))}\Gamma_m.
        \end{align}
        Of course, since $f\in L^\infty(\psi_1)$, this bound also holds true
        for $h=f$.

        \medskip\noindent\textit{Step~2. An upper bound on $B_{m,n}(f)$.}
        Using~\eqref{eq:eqMD1}, we immediately obtain
        \[
        \frac{n^\beta}{(m+n)^\beta} \theta_1^{-m}\left|n^{-\beta}\theta_1^{-n}G_m (Q^n f)-G_m\eta_f\right|\leq \frac{n^\beta}{(m+n)^\beta}\theta_1^{-m}\alpha_n G_m\psi_1
        \]
        and hence
        \begin{align*}
            B_{m,n}(f)&\leq \frac{n^\beta}{(m+n)^\beta} \theta_1^{-m}\alpha_n  \sqrt[p]{\E\left(\left|G_m\psi_1\right|^p\right)}.
        \end{align*}
        But, using~\eqref{eq:bound_Amn} we have
        \[
        \sqrt[p]{\E\left(\left|G_{m}\psi_1-G_0(Q^m\psi_1)\right|^p\right)} = m^\beta \theta_1^\beta A_{0,m}(\psi_1) \leq c_0c_1\,\theta_1^{m-1}\Gamma_0 \, \sqrt[p]{\E(G_0^{(p)}(\psi_2^p))}m^\beta
        \]
        and hence
        \begin{align*}
            \sqrt[p]{\E\left((G_{m}\psi_1)^p\right)}\leq c_0c_1\,\theta_1^{m-1}\Gamma_0 \, \sqrt[p]{\E(G_0^{(p)}(\psi_2^p))}m^\beta  + \sqrt[p]{\E\left(\left|G_0(Q^m\psi_1)\right|^p\right)}
        \end{align*}
        Using~\eqref{eq:eqMD2}, we
        deduce that
        \begin{align*}
            \sqrt[p]{\E\left((G_{m}\psi_1)^p\right)}\leq  c_0c_1\,\theta_1^{m-1}\Gamma_0\, \sqrt[p]{\E(G_0^{(p)}(\psi_2^p))}m^\beta  + c_1m^\beta \theta_1^\beta\,\sqrt[p]{\E\left(\left|G_0\psi_1\right|^p\right)}
        \end{align*}
        Finally, we deduce that
        \begin{align*}
            B_{m,n}(f)&\leq \alpha_n\frac{(nm)^\beta}{(m+n)^\beta}\left( c_0c_1\theta_1^{-1}\Gamma_0  \sqrt[p]{\E(G_0^{(p)}(\psi_2^p))}  + c_1\sqrt[p]{\E\left(\left|G_0\psi_1\right|^p\right)}\right).
        \end{align*}

        \medskip\noindent \textit{Step 3. A upper bound in $L^p$ for the
            Biggins martingale $W^f$.}
        
        Since $\|\eta_f\|_{\psi_1}$, one can use
        inequality~\eqref{eq:bound_A} of Step~1 with
        $h=\eta_f$ to deduce that for all $n,m\geq 0,$
        \begin{align*}
            \sqrt[p]{\E\left(|W^f_{m+n}-W^f_m|^p\right)}&=\sqrt[p]{\E\left(|\theta_1^{-(m+n)}G_{m+n}\eta_f-\theta_1^{-m}G_m(\theta_1^{-n}Q^n\eta_f)|^p\right)}\\
            & =\theta_1^{-(m+n)} \sqrt[p]{\E\left(\left|G_{m+n}(\eta_f)-G_m(Q^n\eta_f)\right|^p\right)}\\
            & \leq \theta_1^{-(m+n)}c_0\sqrt[p]{\E(G_0^{(p)}(\psi_2^p))}\sum_{k=0}^{n-1}\gamma_{m+k}\theta_1^{m+k}\|Q^{n-k-1}\eta_f\|_{\psi_1}\\
            & \leq c_0\theta_1^{-1} \sqrt[p]{\E(G_0^{(p)}(\psi_2^p))} \|\eta_f\|_{\psi_1}\Gamma_m,
        \end{align*}
        where we used that $Q^{n-k-1}\eta_f= \theta_1^{n-k-1}\eta_f$.
        This implies in particular that $(W^f_n)_{n\in\N}$ forms a Cauchy
        sequence in the $L^p(\Omega,\P)$ space. Hence its limit $W^f_\infty$
        belongs to $L^p(\Omega,\P)$ and satisfies, for all $m\geq 0$,
        \begin{align*}
            \sqrt[p]{ \E\left(|W^f_{\infty}-W^f_m|^p\right)}&\leq  c_0\theta_1^{-1} \sqrt[p]{\E(G_0^{(p)}(\psi_2^p))} \|\eta_f\|_{\psi_1}\Gamma_m,
        \end{align*}
        which of course also implies that $W^f$ is a uniformly integrable
        martingale. Thus we have
        \begin{align*}
            C_{m,n}(f) &= \sqrt[p]{\E\left(\left|\frac{n^\beta}{(m+n)^\beta}W_m^f - W_\infty^f \right|^p\right)}\\
            & \leq \left(1 - \frac{n^\beta}{(m+n)^\beta}\right)\left(\sqrt[p]{\E\left(\left|W_m^f-W_0^f\right|^p\right)} + \sqrt[p]{\E\left(\left|W_0^f\right|^p\right)}\right) + \sqrt[p]{\E\left(\left|W_m^f-W_\infty^f\right|^p\right)}\\
            &\leq \left(1 - \frac{n^\beta}{(m+n)^\beta}\right)\left(c_0\theta_1^{-1} \sqrt[p]{\E(G_0^{(p)}(\psi_2^p))} \|\eta_f\|_{\psi_1}\Gamma_0 + \|\eta_f\|_{\psi_1}\sqrt[p]{\E\left(\left|G_0\psi_1\right|^p\right)}\right) + c_0\theta_1^{-1} \sqrt[p]{\E(G_0^{(p)}(\psi_2^p))} \|\eta_f\|_{\psi_1}\Gamma_m.
        \end{align*}
        
        \medskip Steps~1,~2,~3 and~4 and inequality~\eqref{eqStep0} allow us
        to conclude the proof of the main statement of Theorem~\ref{thmMain}.
    \end{proof}

    \begin{remark}
    \label{rem:alphamn}
        In Theorem~\ref{thmMain}, we consider a situation where all the generations have the same type space. The result can actually be applied to models where the state space depends on the generation by simply enriching the type space with the generation number. However, in this situation, one may prefer to adapt the assumption MD($f,p$) and the proof. For instance, consider the case where generation $n$ lives in a state space $\XX_n$, so that $Q$ is a kernel from $\XX_n$ to $\XX_{n+1}$ for each $n\geq 0$, and where, instead of~\eqref{eq:eqMD1}, we have
        \begin{align}
            \label{eq:eqMD1alt}
            \left|n^{-\beta}\theta_1^{-n}\delta_x Q^nf-\eta_f(x)\right|\leq \alpha_{m,n}\psi_1(x),\ \forall n\geq 1,\,x\in \XX_m, 
        \end{align}
        where, for each $m\geq 1$, $\alpha_{m,n}\to 0$ when $n\to+\infty$. Then Theorem~1 holds true with $\alpha_{m,n}$ instead of $\alpha_n$ in~\eqref{eq:eqthmMain}. An application of this consideration is provided in Section~\ref{sec:exainhomogeneous}, where we consider ergodic behaviour along genealogical lineages.
    \end{remark}

    \section{Uniform integrability of the Biggins martingale under a
    $L\log L$ criterion}

\label{secXlogX}

In this section, we provide a $L\log L$  type criterion that can be used to prove that
the Biggins martingale is uniformly integrable and that the convergence of $(G_n(f))_{n\in \mathbb N}$ holds almost surely.

As in the previous sections, we first make an assumption on the time asymptotic behavior of $Q^n$ when $n\to+\infty$.

\textbf{Assumption EB($f,\psi$) (exponential behaviour).}
We say that two measurable functions $f:\XX\to\mathbb R$ and $\psi:\XX\to (0,+\infty)$ satisfy Assumption~EB($f,\psi$) if
$\|f\|_\psi<\infty$ and there exists a sequence of non-negative
numbers $(\alpha_n)_{n\in\N}$ which converges to $0$
and a function $\eta_f:\XX\to \mathbb R$ such that, for some constant $C_1$ and $C_2$, and some $\theta_1>0$,
\begin{align}
    \label{eq:eqMixingLlogL1}
    \left|\theta_1^{-n}\delta_x Q^n f-\eta_f(x)\right|\leq \alpha_n\,C_f\,\psi(x)
    ,\quad \forall n\geq 0,\ x\in \XX;
\end{align}
and 
\begin{align}
    \label{eq:eqMixingLlogL2}
   \theta_1^{-n}\delta_x Q^n \psi\leq c_2\,\psi(x)
    ,\quad\forall n\geq 0,\,x\in \XX;
\end{align}

\medskip 

For any function $f:\XX\to\mathbb R_+$, we define 
\[
X_k^f(x):=G^x_k(f)-\mathbb E(G^x_k(f)),
\]
where $G_x^k$ is the weighted empirical distribution of the $k$-th generation of the weighted branching process, with initial position $\delta_x$.

Our second assumption relates to a $L\log L$ type criterion, as we explain below. Let $p\in(1,2]$ and recall that $Q^{(p)}$ and $G^{(p)}$ are defined in the previous section.

\medskip\noindent \textbf{Assumption H($f,k$).} We say that a measurable function $f:\XX\to\mathbb R_+$ and an integer $k\geq 1$ satisfy Assumption~H($f,k$) if  there exists $\theta_1,\rho>0$ such that
\begin{align*}
    \begin{cases}
        &\sum_{n\geq 1} \theta_1^{-n}\mathbb E(G_0)Q^ng^{(1)}_{\rho,k,n}<\infty,\\
        &\sum_{n\geq 1} \theta_1^{-pn}\mathbb E(G_0^{(p)})(Q^{(p)})^ng^{(2)}_{\rho,k,n}<\infty,
    \end{cases}    
\end{align*}
where 
\begin{align*}
    &g^{(1)}_{\rho,k,n}(x)=\mathbb E(|X_k^f(x)|\mathbf 1_{|X_k^f(x)|> \rho^n}),\\
    &g^{(2)}_{\rho,k,n}(x)=\mathbb E(|X_k^f(x)|^p\mathbf 1_{|X_k^f(x)|\leq \rho^n}).
\end{align*}

In the following two results, we recall that the Biggins martingale $W^f$ is the positive martingale defined by $W^f_n=\theta_0^{-n}\bar\mu_n(\eta_f)$ for all $n\geq 0$, where $\eta_f$ is from Assumption EB($f,\psi$).

\begin{theorem}
    \label{thm:unifBiggins}
    Let $f,\psi$ be such that EB($f,\psi$) holds true. If Assumption H($\eta_f,1$) holds true  and $\mathbb E(G_0(\eta_f))<\infty$, then the non-negative martingale $W^f$ is uniformly integrable.
\end{theorem}

\begin{theorem}
\label{thm:asLlogL}
Assume that there exist some functions $f,\psi$ with $f\geq \psi$  and such that EB($f,\psi$) and H($f,k$) hold true for all $k\geq 1$.
Assume in addition that $G_0(\psi)<\infty$ holds true. 
Then $(\theta_1^{-n}G_n(f))_{n\in\mathbb N}$ converges almost surely to $W^f_\infty$.
Furthermore, for all function $g:\XX\to \mathbb R$ such that EB($g,\psi$) and H($g,k$) hold true for all $k\geq 1$, $(\theta_1^{-n}G_n(g))_{n\in\mathbb N}$ converges almost surely to $W^g_\infty$.
\end{theorem}

Note that, combining the first theorem, the second one and Scheffé's Lemma, we deduce that, if the assumptions of the two theorems hold true, then $(\theta_0^{-n}\bar\mu_n(f))_{n\in\mathbb N}$ converges in $L^1$ to $W_\infty^f$. 

Let us also mention that our proof follows partly the method developed Asmussen and Hering~\cite{AsmussenHering1976} and by Jagers in~\cite{Jagers1989}. Our main contributions is the development of new results in the setting of weighted branching processes and to provide refinement  to adapt the  convergence result to situations where the semi-group $Q$, defined in~\eqref{eq:intro-many-to-one}, does not converge in total variation norm, admits  a non-exponential renormalizing factor, and when the limit is not a one dimensional operator.

Before turning to the proof of the two theorems, we start with a Lemma which relates the above condition to $L\log L$ type criteria. 
We define for $a\geq 1$ the function 
\[x\in \R_+ \mapsto \loga(x) = \begin{cases}
    \left(\frac{a}{\mathrm{e}}\right)^a x,& \text{ if } x<\mathrm{e}^a,\\ \log(x)^a,& \text{ if }x\geq \mathrm{e}^a.
\end{cases}\]

\begin{lemma}
\label{lem:lemma3} Assume that  there exist two constants $\theta_1,\theta_2>0$ such that $\theta_1^p>\theta_2$, two functions $\psi,\varphi: \XX\longrightarrow \mathbb R$, a  non-negative measure $\pi$, a non-negative sequence $(\gamma_n)_{n\geq 0}$, such that, for all $g:E\to \mathbb R_+$ and $n\geq 1$,
    \begin{align}
        &\theta_1^{-n}\mathbb E(G_0)Q^n g\leq \pi g+\gamma_n\|g\|_\varphi\label{eq:lem3eq1}\\
        &\theta_2^{-n}\mathbb E(G_0^{(p)})(Q^{(p)})^n g\leq \pi g+\gamma_n\|g\|_\varphi.\label{eq:lem3eq2}
    \end{align}
    Assume also that, for some $\alpha \geq 1$ such that $\sum_{n\geq 1} \gamma_n/n^\alpha <+\infty$ and some $k\in \mathbb N$ such that $Q\psi \in L^\infty(\psi)$,
    \begin{align*}
        \forall x\in \XX,\,\mathbb E\left(G_1^x(\psi) \logalpha (G_1^x(\psi))\right) \leq \varphi(x),\quad
        \psi(x) \logalpha(\psi(x))\leq \varphi(x),\\
        \forall 0\leq j \leq k-1,\,\pi(Q^j\varphi)<+\infty,\,\pi\left((Q^{(p)})^j\psi\right)<+\infty,
    \end{align*}
    and if either $(\gamma_n)_{n\in\NN}$ is null or
    \begin{equation}
    \label{eq:lemma3ifgammanonnull}
     \forall 0\leq j\leq k-1,\:\|Q^j\varphi\|_\varphi <+\infty,\,\|(Q^{(p)})^j\psi\|_\varphi<+\infty,
    \end{equation}
    Then, for all function $f\in L^\infty(\psi)$, Assumption~H($f,k$) holds true.
\end{lemma}

\begin{remark}
    The second set of assumptions in Lemma~\ref{lem:lemma3} is used to prove that
    for some function $f:\XX\to\mathbb R_+$, $k\geq 1$ and $\alpha\geq 0$ such that $\sum_{n\geq 1} \gamma_n/n^{1+\alpha}<\infty$,
    \begin{align*}
        &\int_\XX\pi(\mathrm dx)\,\mathbb E(X_k^f\log_+ X_k^f\mid G_0=\delta_x)<\infty,\\
        &\|x\mapsto \mathbb E(|X_k^f\log_+^{1+\alpha} X_k^f\mid G_0=\delta_x)\|<\infty.
    \end{align*}
    It is thus possible to obtain a possibly more general criterion by replacing the second set of assumptions of Lemma~\ref{lem:lemma3} with this, in a similar fashion as in~\cite{bansaye2025strong}. However, in our case, there doesn't seem to be situations where this is easy to prove by a direct way and we thus prefer to keep the more practical formulation of Lemma~\ref{lem:lemma3}.
\end{remark}

Let us now provide and comment on some sufficient conditions for~\eqref{eq:lem3eq1} and~\eqref{eq:lem3eq2}.

\begin{example}
To keep the situation simple, let us consider the situation where, for all $x\in \XX$, the law of $(u_i,x_i)_{i\in\mathbb N}$ under $K(x,\cdot)$ is such that the $u_i$ are in $[0,1]$ almost surely (this includes the classical branching process case, where $u_i\equiv 1$).
Let us assume that Assumption~EB($f,\psi$) is in force with $\theta_1>1$.

We obtain that~\eqref{eq:lem3eq1} and~\eqref{eq:lem3eq2} hold in the following situations.
\begin{enumerate}
    \item It holds if $\mathbb E(G_0)\varphi<\infty$ and, for all function measurable $g:\XX\to \mathbb R$ with $\|g\|_\varphi<\infty$, we have
    \begin{equation*}
        \left|\theta_1^{-n} \delta_x Q g -\eta(x)\pi(g)\right|\leq \gamma_n\|g\|_\varphi\varphi(x),\quad\forall n\geq 0,\ x\in \XX,
    \end{equation*}
    with $\|\eta\|_\varphi<\infty$ and $\pi(\varphi)<\infty$. Indeed, in this case, $\pi Q=\theta_1 Q$ and hence, setting e.g. $\theta_2=\theta_1^{p/2}<\theta_1^p$, we have, for all non-negative $g$,
    \begin{align*}
        &\theta_1^{-n}\mathbb E(G_0)Q^n g\leq \mathbb E(G_0)\eta \,\pi(g)+\gamma_n \|g\|_\varphi,\\
        &\theta_2^{-n}\mathbb E(G_0)(Q^{(p)})^n g\leq \theta_1^{-n}\mathbb E(G_0)Q^n g\leq \mathbb E(G_0)\eta \,\pi(g)+\gamma_n \|g\|_\varphi.
    \end{align*}
    \item It holds if there exists a measure $\pi$ on $\XX$ such that $\pi Q\leq \theta_1 \pi$ and  $\mathbb E(G_0)\leq \pi$, in which case one can take $\gamma_n\equiv 0$. In particular, this shows that, if $\pi Q\leq \theta_1 \pi$, then it holds true for $G_0=\delta_{x_0}$ for $\pi$-almost all $x_0$. 

    Compared to the previous case, this is particularly useful when convergence in~\eqref{eq:eqMixingLlogL1} does not hold in the strong weighted total variation norm $\|\cdot\|_\psi$. This is the case for instance if $\theta_1^{-n}\delta_x Q^n$ converges to $\eta(x)\pi$ uniformly in $W_1$-Wasserstein distance (see e.g.~\cite{ChampagnatEtal2025}).

    Also, since $(\gamma_n)_{n\in \mathbb N}$ is null and contrarily to the previous case, the $\psi$-norm of the function $x\mapsto\mathbb E\left(G_1^x(\psi) \logalpha (G_1^x(\psi))\right)$ may be infinite: the $L\log L$-type criterion is only required on average (relatively to $\pi$).
    
    \item As an extension to the previous case, if there exists a measure $\pi$ on $\XX$ such that $\pi Q\leq \theta_1 \pi$ and  such that $\mathbb E(G_0)Q^{n_0}$ admits a bounded density with respect to $\pi$ (which may be the case when $Q$ is regularizing), then the weighting branching process with initial position $G_{n_0}$ (note that $\mathbb E(G_{n_0})=\mathbb E(G_0)Q^{n_0}$) satisfies the two first conditions of Lemma~\ref{lem:lemma3} with $\gamma_n\equiv 0$.

    As in the previous case, this allows for situation where the $\psi$-norm of the function $x\mapsto\mathbb E\left(G_1^x(\psi) \logalpha (G_1^x(\psi))\right)$ is infinite.

    \item In the two previous cases, we assumed that $\mathbb E(G_0)Q^{n_0}$ admits a bounded density with respect to $\pi$, which may be inconvenient in some situations. To see this, assume that that $\pi(\mathrm dx)=f_\pi(x)\Lambda(\mathrm dx)$ for some reference measure $\Lambda$ on $\XX$ and with $f_\pi(x)>0$ but $\inf f_\pi=0$, and that $\mathbb E(G_0)Q^{n_0}(\mathrm dx)=f_{n_0}(x)\Lambda(\mathrm dx)$ with $f_{n_0}$ a non-negative measurable function (which is typically the case when $Q$ is regularising). This setting corresponds to the previous case if and only if $\sup f_{n_0}/f_{\pi}<+\infty$. 

    In the case where $\sup f_{n_0}/f_{\pi}=+\infty$, let $\XX_\ell:=\{f_{n_0}/f_{\pi}\leq \ell\}$ for all $\ell\geq 1$, so that $\XX=\cup_{\ell\geq 1} \XX_\ell$. Then, for any fixed $\ell$, the process with initial position $G_{n_0}(\cdot\cap (\XX\setminus \XX_\ell))$ satisfies conditions~\eqref{eq:lem3eq2} and~\eqref{eq:lem3eq2} with $\gamma_n\equiv 0$, since it enters the setting of the previous case. Hence if the $L\log L$ criterion is satisfied (only for averages with respect to $\pi$ since $\gamma_n\equiv 0$), then this process rescaled by $\theta_1^{-n}$ converges almost surely according to Theorems~\ref{thm:unifBiggins} and~\ref{thm:asLlogL}.

    We deduce that the original process converges almost surely if the following condition is satisfied: 
    \begin{align*}
        \mathbb P\left(\exists \ell\geq 1,\ G_{n_0}(\XX\setminus \XX_\ell)=0\right)=1,
    \end{align*}
    which is the case for instance if,
    almost surely,
    the number of individuals in generation $n_0$ is finite. 
    \item Finally, in some situations the process may lack the existence of reference measure $\pi$, in which case the previous sufficient conditions are not helpfull. However,  if~\eqref{eq:lemma3ifgammanonnull} holds true for $\pi=0$ and $\alpha>1$, then one can take $(\gamma_n)_{n\in\mathbb N}$ constant and Conditions~\eqref{eq:lem3eq1} and ~\eqref{eq:lem3eq2} reduce to
    \[
    \theta_1^{-n}\mathbb E(G_0)Q^n g\leq C\|g\|_\varphi,
    \]
    which does not require the existence of a reference measure. This criterion should be particularly helpful for time-inhomogeneous processes, periodic processes or processes for which only some marginals/projections converge.
\end{enumerate}
\end{example}

\bigskip

Let us now prove the main results of this section. We first prove Lemma~\ref{lem:lemma3}, then Theorem~\ref{thm:unifBiggins} and~\ref{thm:asLlogL}.

\medskip
The function $\log^{(a)}$ was studied for $a=1$ in \cite{AsmussenHering1976}, in the context of branching processes without weights. We give, without proof, the following extension.

\begin{lemma}\label{lem:logalpha}
For all $a\geq 1$, the function $x\mapsto \loga (x)$ is concave, and the application $x\mapsto x\loga (x)$ is increasing and convex. In addition, if $X_1,\ldots ,X_N$ is a sequence of independent random variables, and $S=X_1+\ldots+X_N,$ then 
\[\mathbb E(S\loga (S)) \leq \mathbb E(S) \loga (\mathbb E(S)) + \sum_{i=1}^N \mathbb E\left(X_i \loga (X_i)\right),\]
and for all $a\geq 1$ and $p>0$ there exists $C_{a,p}>0$ such that for all $x\geq 0$
\[\loga(x)\leq C_{a,p} (1 + x^p).\]
\end{lemma}

\begin{proof}[Proof of Lemma~\ref{lem:lemma3}]
We start by showing in Step~1 that 
\[\int_\XX\pi(dx)\mathbb E\left(|X_k^f(x)|\log^{(1)}\left(|X_k^f(x)|\right)\right)<+\infty,\]
and that, for $\alpha\geq 1$ such that $\sum_{n\geq 1} \gamma_n/n^{\alpha}<\infty$, we have
\[\left\|x\mapsto \mathbb E\left(|X_k^f(x)|\logalpha \left(|X_k^f(x)|\right)\right)\right\|<+\infty.\]
Then we prove that Assumption~$H(f,k)$ holds in Step~2.

\medskip\noindent{\textit{Step~1}} It is enough to prove that for $a\in \{1,\alpha\}$ there exists $K_1>0$ such that,
\begin{equation}\label{eq:LemmeLlogL}
    \mathbb E\left(|X_k^f(x)|\loga \left(|X_k^f(x)|\right)\right) \leq \mathbb E\left(G_k^x(\psi) \loga (G_k^x(\psi))\right) + K_1 \psi(x)\loga (\psi(x)).
\end{equation}

Indeed, let us assume first that \eqref{eq:LemmeLlogL} holds true. Then, by the branching property we have for all $1\leq j \leq k$,
\[G_j^x(\psi) = \sum_{|e|=j-1}w_eG_1^{X_e}(\psi),\]
where for $e\neq e'$, $G_1^{X_e}$ and $G_1^{X_{e'}}$ are independent conditioned on $\mathcal F_{j-1}$. Thus, by Lemma~\ref{lem:logalpha}, 
\begin{align}\label{eq:bound_Step1}
  \mathbb E\left.\left(G_j^x(\psi)\loga (G_j^x(\psi))\, \right|\, \mathcal F_{j-1}\right) & \leq \mathbb E(G_j^x(\psi)\mid \mathcal F_{j-1}) \loga\left(\mathbb E(G_j^x(\psi)\mid \mathcal F_{j-1})\right) +  \sum_{|e|=j-1} \mathbb E\left.\left(w_eG_1^{X_e}\loga \left(w_eG_1^{X_e}(\psi)\right) \,\right|\,\mathcal F_{j-1}\right) \nonumber\\
  &\leq K_2 G_{j-1}(\psi)\loga (G_1^x(\psi))+\sum_{|e|=j-1} \mathbb E\left.\left(w_eG_1^{X_e}\loga \left(w_eG_1^{X_e}(\psi)\right) \,\right|\,\mathcal F_{j-1}\right),
  \end{align}
for some $K_2>0$. The sum on the right-hand side is bounded according to Lemma~\ref{lem:logalpha} and by assumption by

\begin{align*}
    \sum_{|e|=j-1} C_{a,p-1}\mathbb E\left.\left(w_eG_1^{X_e}(\psi)+w_e^pG_1^{X_e}(\psi)\right| \mathcal F_{j-1}\right)&+\sum_{|e|=j-1}\mathbb E\left.\left(w_eG_1^{X_e}(\psi)\loga \left(G_1^{X_e}(\psi)\right) \right| \mathcal F_{j-1}\right)\\
    &\leq \sum_{|e|=j-1} C_{\alpha,p-1}(w_e+w_e^p)Q\psi(X_e) + \sum_{|e|=j-1} w_e\varphi(X_e)\\
    &\leq K_3 \left(G_{j-1}^x(\psi)+G_{j-1}^{(p),x}(\psi) +G_{j-1}^x(\varphi)\right),
\end{align*}
for some $K_3>0$.

Thus, by taking expectation in~\eqref{eq:bound_Step1} we obtain that there exists $K_4>0$
\[\mathbb E\left(G_j^x(\psi)\loga (G_j^x(\psi))\right) \leq \mathbb E\left(G_{j-1}^x(\psi)\loga (G_{j-1}^x(\psi)\right) + K_4 \left(\psi(x) + \left(Q^{(p)}\right)^{j-1}\psi(x) + Q^{j-1}\varphi(x)\right),\]
where we used that $Q\psi \in L^\infty(\psi)$ by assumption. An iteration argument combined with~\eqref{eq:LemmeLlogL} yields
\begin{align*}\mathbb E\left(\left|X_k^f(x)\right|\loga \left(\left|X_k^f(x)\right|\right)\right) &\leq K_5\left(\sum_{j=0}^{k-1}\left(\left(Q^{(p)}\right)^j\psi(x) + Q^j\varphi(x)\right) + \psi(x)\loga(\psi(x))\right)\\
& \leq K\sum_{j=0}^{k-1}\left(\left(Q^{(p)}\right)^j\psi(x) + Q^j\varphi(x)\right),
\end{align*}
for $K_5,K>0$, where we used that $a\in \{1,\alpha\}$, and that $\psi\logone \psi \leq \psi\logalpha \psi \leq \varphi = Q^0\varphi'$, by assumption.

Hence, setting $a=1$ and integrating with respect to $\pi$, we have 
\begin{align*}
    \int_\XX \pi(dx)\mathbb E\left(\left|X_k^f(x)\right| \logone\left(\left|X_k^f(x)\right|\right)\right) \leq K\sum_{j=0}^{k-1}\left(\pi\left((Q^{(p)})^j\psi\right)  + \pi\left(Q^j\varphi\right)\right)
    <+\infty.
\end{align*}
Similarly, since the norm $\|\cdot\|$ is lattice, we have considering $a=\alpha$,
\[\left\|x\mapsto \mathbb E\left(\left|X_k^f(x)\right|\logalpha\left(\left|X_k^f(x)\right|\right)\right)\right\| \leq K \sum_{j=0}^{k-1}\left(\left\|(Q^{(p)})^j\psi\right\| + \left\|Q^j\varphi\right\|\right)<+\infty.\]
Thus, to conclude Step~1, we prove \eqref{eq:LemmeLlogL}.
Indeed, for $1\leq j\leq k$, we have that almost surely
\[|X_{j}^f(x)|\leq G_{j}^x(\psi) + \mathbb E(G_{j}^x(\psi))=G_{j}^x(\psi) + \delta_xQ^{j}\psi \leq G_{j}^x(\psi) + C^{j}\psi(x),\]
for some $C>0$, since $Q\psi\in L^\infty(\psi)$ by assumption.

Hence, for $a\in \{1,\alpha\}$
\begin{align*}
    \mathbb E\left(|X_{j}^f(x)| \loga\left(\left.\left|X_{j}^f(x)\right|\right)\,\right|\,\mathcal F_{j-1}\right)
    &\leq \mathbb E\left(\left.G_{j}^x(\psi)\loga(G_{j}^x(\psi))\,\right|\, \mathcal F_{j-1}\right) + \mathbb E\left(G_{j}^x(\psi)\mid \mathcal F_{j-1}\right)\loga(C^{j}\psi(x)) \\
    &+ C^{j}\psi(x) \mathbb E\left(\left.\loga(G_{j}^x(\psi)\,\right|\, \mathcal F_{j-1}\right) + C^{j}\psi(x) \loga(C^{j}\psi(x)).
\end{align*}

By Jensen's inequality and using that by the branching property $\mathbb E(G_{j}^x(\psi)\mid \mathcal F_{j-1}) = G_{j-1}^x(Q\psi) \leq CG_{j-1}^x(\psi)$ by assumption, we have,
\begin{align*}
    \mathbb E\left(\left|X_{j}^f(x)\right| \loga\left(\left|X_{j}^f(x)\right|\right)\right) 
    &\leq \mathbb E\left(G_{j}^x(\psi)\loga(G_{j}^x(\psi))\right) + 3C^{j}\psi(x) \log^*(C^{j}\psi(x))\\
    &\leq \mathbb E\left(G_{j}^x(\psi)\loga(G_{j}^x(\psi))\right) + K_1\psi(x) \loga(\psi(x)),
\end{align*}
for $K_1>0$, and~\eqref{eq:LemmeLlogL} follows.

\medskip\noindent Step~2. Let us now prove that Assumption $H(f,k)$ holds.  We have by assumption that for all $\rho >1$, 
\begin{align*}
    \sum_{n\geq 1} \theta_1^{-n} \mathbb E(G_0)Q^n g_{\rho, k, n}^{(1)}(x) \leq \sum_{n\geq 1} \left(\pi \left(g_{\rho, k, n}^{(1)}\right)+\gamma_n \|g_{\rho,k,n}^{(1)}\|\right).
\end{align*}

The first term on the right-hand side verifies,
\begin{align*}
    \sum_{n\geq 1} \pi\left(g_{\rho,k,n}^{(1)}\right)&= \int_\XX\sum_{n\geq 1}\mathbb E\left(|X_k^f(x)|\mathbf 1_{|X_k^f(x)| > \rho^n} \right) \pi(dx)\\
    &\leq \int_\XX \mathbb E\left(|X_k^f(x)| \sum_{n\geq 1} \mathbf 1_{\frac{\logone(|X_k^f(x)|)}{\log(\rho)}>n}\right)\pi(dx)\\
    &\leq \frac{1}{\log(\rho)}\int_\XX \mathbb E\left(|X_k^f(x)| \logone(|X_k|)\right)\pi(dx),
\end{align*}
which is finite by Step~1, where in the first inequality we used that $x\mapsto \logone (x)$ is non-decreasing and that $0\leq n\log (\rho) \leq \logone(\rho^n)$. For the second term we have,
\begin{align*}
    \sum_{n\geq 1}\gamma_n \|g_{\rho, k,n}^{(1)}\| &= \sum_{n\geq 1} \gamma_n \left\|\mathbb E\left(|X_k^f(\cdot)| \mathbf 1_{|X_k^f(\cdot)|> \rho^n}\right)\right\|\\
    & \leq \sum_{n\geq 1} \gamma_n \dfrac{\left\|\mathbb E\left(|X_k^f(\cdot)|\logalpha(|X_k^f(\cdot)|)\right)\right\|}{\logalpha(\rho^n)}\\
    &= \left\|\mathbb E\left(|X_k^f(\cdot)|\logalpha(|X_k^f(\cdot)|)\right)\right\|\left(\dfrac{\mathrm{e}^\alpha}{\alpha^\alpha}\sum_{\substack{n\geq 1\\ \rho^n < \mathrm{e}^\alpha}} \dfrac{\gamma_n}{\rho^n}+\dfrac{1}{\log(\rho)}\sum_{\substack{n\geq 1\\ \rho^n\geq \mathrm{e}^\alpha}}\dfrac{\gamma_n}{n^\alpha} \right),
\end{align*}
which is finite according to Step~1 and since $\alpha\geq 1$ is such that $\sum_{n} \gamma_n/n^\alpha$ is finite by assumption. In the inequality we have used Markov's inequality since $x\mapsto \logalpha(x)$ is non-decreasing.

Similarly, for the second term in Assumption~$H(f,k)$, we have by defining 
\[\rho := \left(\dfrac{\theta_1^p}{\theta_2}\right)^{\frac{1}{p-1}}>1,\]
we have
\begin{equation}\label{eq:lem3_decomposition}
    \sum_{n\geq 1} \theta_1^{-pn} \mathbb E(G_0^{(p)})(Q^{(p)})^n g_{\rho, k, n}^{(2)} \leq \sum_{n\geq 1} \left(\dfrac{\theta_2}{\theta_1^p}\right)^n\pi \left(g_{\rho, k, n}^{(2)}\right)+\sum_{n\geq 1} \left(\dfrac{\theta_2}{\theta_1^p}\right)^n\gamma_n \|g_{\rho,k,n}^{(2)}\|.
\end{equation}

For the first term on the right-hand side we have
\begin{align*}
    \sum_{n\geq 1}\left(\dfrac{\theta_2}{\theta_1^p}\right)^n \pi\left(g_{\rho, k, n}^{(2)}\right)&= \int_\XX \mathbb E\left(|X_k^f(x)|^p \sum_{n\geq 1}\left(\dfrac{\theta_2}{\theta_1^p}\right)^n\mathbf 1_{|X_k^f(x)| \leq \rho^n}\right)\pi(dx)\\
    &\leq \int_\XX \mathbb E\left(|X_k^f(x)|^p \left(\sum_{\substack{n\geq 1\\ \rho^n<\mathrm e}}\left(\frac{\theta_2}{\theta_1^p}\right)^n \mathbf 1_{|X_k^f(x)| \leq \rho^n}+\sum_{\substack{n\geq 1\\ \rho^n\geq\mathrm e}}\left(\frac{\theta_2}{\theta_1^p}\right)^n\mathbf 1_{\logone (|X_k^f(x)|\leq n \log (\rho)}\right)\right)\pi(dx)\\
    &\leq K' \int_\XX \mathbb E\left(|X_k^f(x)|^p \sum_{n\geq \frac{\logone(|X_k^f(x)|)}{\log(\rho)}} \left(\dfrac{\theta_2}{\theta_1^p}\right)^n\right)\pi(dx)\\
    &\leq\dfrac{K'\theta_1^p}{\theta_1^p-\theta_2} \int_\XX \mathbb E\left(|X_k^f(x)|^p \left(\dfrac{\theta_2}{\theta_1^p}\right)^ {\frac{\logone(|X_k^f(x)|)}{\log(\rho)}} \right) \pi(dx)\\
    &\leq\dfrac{K'\theta_1^p}{\theta_1^p-\theta_2} \int_\XX \mathbb E\left(|X_k^f(x)|^p \left(\dfrac{\theta_2}{\theta_1^p}\right)^ {\frac{\logone(|X_k^f(x)|)}{\log(\rho)}} \right) \pi(dx)
\end{align*}
for some $K'>0$.

If we now define the function 
\[\expone: x\in \mathbb R_+\mapsto \expone(x) = \begin{cases}x, &\text{ if }x\geq \mathrm e,\\ \exp(x/\mathrm e),&\text{ if }x<\mathrm e,\end{cases}\]
we have that 
\[\left(\dfrac{\theta_2}{\theta_1^p}\right)^ {\frac{\logone(|X_k^f(x)|)}{\log(\rho)}} = \expone\left(|X_k^f(x)|\right)^{\frac{\log \left(\frac{\theta_2}{\theta_1^p}\right)}{\log(\rho)}} = \expone\left(|X_k^f(x)|\right)^{1-p},\]
where we used the definition of $\rho$.

Hence,
\begin{align*}
     \sum_{n\geq 1}\left(\dfrac{\theta_2}{\theta_1^p}\right)^n \pi\left(g_{\rho, k, n}^{(2)}\right)&\leq \dfrac{K'\theta_1^p}{\theta_1^p-\theta_2} \int_\XX \mathbb E\left(|X_k^f(x)|^p \expone\left(|X_k^f(x)|\right)^{1-p}\right) \pi(dx)\\
     & = \dfrac{K'\theta_1^p}{\theta_1^p-\theta_2} \int_\XX \left[\mathrm e^p\mathbb E\left( \left(\dfrac{|X_k^f(x)|}{\mathrm e}\right)^p \exp\left(\dfrac{|X_k^f(x)|}{\mathrm e}\right)^{1-p}\mathbf 1_{|X_k^f(x)|<\mathrm e}\right) + \mathbb E\left(|X_k^f(x)| \mathbf 1_{|X_k^f(x)|\geq \mathrm e}\right)\right]\pi(dx)\\
     &\leq \dfrac{K'\theta_1^p}{\theta_1^p-\theta_2} \int_\XX \left[\mathrm e^p\mathbb E\left(\dfrac{|X_k^f(x)|}{\mathrm e} \mathrm e^{1-p}\mathbf 1_{|X_k^f(x)|<\mathrm e}\right) + \mathbb E\left(|X_k^f(x)| \mathbf 1_{|X_k^f(x)|\geq \mathrm e}\right)\right]\pi(dx)\\
     &=\dfrac{K'\theta_1^p}{\theta_1^p-\theta_2} \int_\XX \mathbb E\left(|X_k^f(x)|\right)\pi(dx)\\
     &\leq \dfrac{2K'\theta_1^p}{\theta_1^p-\theta_2} \int_\XX \mathbb E\left(G_k^x(\psi)\right)\pi(dx)\\
     &= \dfrac{2K'\theta_1^p}{\theta_1^p-\theta_2} \pi\left(Q^k\psi\right)\\
     & \leq \dfrac{2K'\theta_1^pC^k}{\theta_1^p-\theta_2}\pi(\psi)\\
     &<+\infty
\end{align*}
for some constant $C>0$, since $Q\psi \in L^\infty(\psi)$. 

Finally, for the second term on the right-hand side of \eqref{eq:lem3_decomposition}, we have that using the definition of $\rho$,

\begin{align*}
    \sum_{n\geq 1}\left(\dfrac{\theta_2}{\theta_1^p}\right)^n \gamma_n \left\|\mathbb E\left( |X_k^f(\cdot)|^p \mathbf 1_{|X_k^f(\cdot)|\leq \left(\frac{\theta_1^p}{\theta_2}\right)^{\frac{n}{p-1}}}\right)\right\|
    &\leq \sum_{n\geq 1}\left(\dfrac{\theta_2}{\theta_1^p}\right)^n \gamma_n \left\|\mathbb E\left( |X_k^f(\cdot)|^p \mathbf 1_{\exp \left(\frac{\alpha}{p-1}\right)\leq |X_k^f(\cdot)|\leq \left(\frac{\theta_1^p}{\theta_2}\right)^{\frac{n}{p-1}}}\right)\right\|\\
    &\qquad + \sum_{n\geq 1}\left(\dfrac{\theta_2}{\theta_1^p}\right)^n \gamma_n \left\|\mathbb E\left( |X_k^f(\cdot)|^p\mathbf 1_{|X_k^f(\cdot)|\leq \exp\left(\frac{\alpha}{p-1}\right)}\right)\right\|\\
    &\leq \sum_{n\geq 1}\left(\dfrac{\theta_2}{\theta_1^p}\right)^n \gamma_n \left\|\mathbb E\left( |X_k^f(\cdot)|^p \mathbf 1_{\exp\left(\frac{\alpha}{p-1}\right)\leq |X_k^f(\cdot)|\leq \left(\frac{\theta_1^p}{\theta_2}\right)^{\frac{n}{p-1}}}\right)\right\|\\
    &\qquad + \exp(\alpha)\left\|\mathbb E\left( |X_k^f(\cdot)| \right)\right\|\,\sum_{n\geq 1}\left(\dfrac{\theta_2}{\theta_1^p}\right)^n \gamma_n 
\end{align*}
By assumption, $\left\|\mathbb E\left( |X_k^f(\cdot)| \right)\right\|<\infty$ and $\sum_{n\geq 1}\left(\dfrac{\theta_2}{\theta_1^p}\right)^n \gamma_n<\infty$. 
In addition, using that $x\mapsto \frac{x^{p-1}}{\logalpha (x)}$ is increasing on $\left[\exp\left(\frac{\alpha}{p-1}\right),+\infty\right)$,
\begin{align*}
    \sum_{n\geq 1}\left(\dfrac{\theta_2}{\theta_1^p}\right)^n \gamma_n \left\|\mathbb E\left( |X_k^f(\cdot)|^p \mathbf 1_{\exp\left(\frac{\alpha}{p-1}\right)\leq |X_k^f(\cdot)|\leq \left(\frac{\theta_1^p}{\theta_2}\right)^{\frac{n}{p-1}}}\right)\right\|
    &\leq \sum_{n\geq 1} \gamma_n \left(\dfrac{\theta_2}{\theta_1^p}\right)^n\left\|\mathbb E\left( |X_k^f(\cdot)|\logalpha( |X_k^f(\cdot)|)\,\frac{\left(\frac{\theta_1^p}{\theta_2}\right)^{n}}{\left[\log\left(\left(\frac{\theta_1^p}{\theta_2}\right)^{\frac{n}{p-1}}\right)\right]^\alpha}\right)\right\|\\
    &=\frac{(p-1)^\alpha\left\|\mathbb E\left(|X_k^f(\cdot)|\logalpha(|X_k^f(\cdot)|)\right)\right\|}{\log\left(\frac{\theta_1^p}{\theta_2}\right)}\sum_{n\geq 1} \frac{\gamma_n}{n^\alpha}\\
    &<+\infty
\end{align*}
by assumption. 

This concludes the proof.
\end{proof}

\begin{proof}[Proof of Theorem~\ref{thm:unifBiggins}]
     We write
    \begin{align*}
        W^f_{n+1}-W^f_n=W^f_{n+1}-\mathbb E(W^f_{n+1}\mid\cF_n)=\theta_1^{-n-1}\sum_{e\in\mathbb N^n}^\infty w_i X_{e,1}^{\eta_f}=A_{n+1}+B_{n+1},
    \end{align*}
    where the $X_{e,1}^{\eta_f}$, $e\in\mathbb N^n$, are independent random variables, each with the same distribution  as $X_1^{\eta_f}(X_e)$, and where
    \begin{align}
        \label{equation9}
        A_{n+1}&=\theta_1^{-n-1}\sum_{e\in\mathbb N^n} w_e X_{e,1}^{\eta_f}\1_{|X_{e,1}^{\eta_f}|\leq \rho^n}-\mathbb E(\theta_1^{-n-1}\sum_{e\in\mathbb N^n} w_e X_{e,1}^{\eta_f}\1_{|X_{e,1}^{\eta_f}|\leq \rho^n}\mid \mathcal F_n),\\ 
        B_{n+1}&=\theta_1^{-n-1}\sum_{e\in\mathbb N^n} w_e X_{e,1}^{\eta_f}\1_{|X_{e,1}^{\eta_f}|> \rho ^n}-\mathbb E(\theta_1^{-n-1}\sum_{e\in\mathbb N^n} w_e X_{e,1}^{\eta_f}\1_{|X_{e,1}^{\eta_f}|> \rho^n}\mid \mathcal F_n).
    \end{align}    
 We observe that, for some constant $c_p\geq 1$ (using \cite{Chatterji1969}),
  \begin{align*}
     \mathbb E(|A_{n+1}|^p)&= \theta_1^{-pn-p}\mathbb E(|\sum_{e\in\mathbb N^n} w_e X_1^{\eta_f}(X_e)\1_{|X_1^{\eta_f}(X_e)|\leq \rho^n}-\mathbb E(X_1^{\eta_f}(X_e)\1_{|X_1^{\eta_f}(X_e)|\leq \rho^n}\mid \mathcal F_n)|^p)\\
     &\leq c_p  \theta_1^{-pn-p}\mathbb E( \sum_{e\in\mathbb N^n}  w_e^p \mathbb E(|X_1^{\eta_f}(X_e)|^p\1_{|X_1^{\eta_f}(X_e)|\leq \rho^n}\mid\mathcal F_n))\\
     &= c_p  \theta_1^{-pn-p} \mathbb E( \sum_{e\in\mathbb N^n} w_e^p g^{(2)}_{\rho,k,n}(X_e))\\
     &=c_p \theta_1^{-pn-p} \mathbb E(G_0^{(p)})(Q^{(p)})^n(g^{(2)}_{\rho,k,n}).
 \end{align*}
 We deduce, using Assumption~H($\eta_f,1$), that $\sum_{n\geq 0} \mathbb E(|A_n|^p)<\infty$, so that the martingale defined by $\sum_{k=0}^n A_k$ is bounded in $L^p$ and is thus uniformly integrable.
    
    Similarly, we obtain, using again Assumption~H($\eta_f,1$), that
    \begin{align*}
        \sum_{n\geq 0}\mathbb E\left|B_{n+1}\right|
        &\leq 2 \sum_{n\geq 0} \theta_1^{-p{n+1}}\mathbb E\left(\sum_{e\in\mathbb N^n} w_e g^1_{n,1}(x)\right)
        = \sum_{n\geq 0} \mathbb E(G_0)Q^ng^1_{n,1} <\infty.
    \end{align*}
    This implies that  the martingale defined by $\sum_{k= 0}^n B_k$ is uniformly integrable.
    
    We conclude that the martingale defined by $W^f_n-W^f_0$ is uniformly integrable, and, since $W^f_0=G_0(\eta_f)$ is integrable by assumption, that $W^f$ is uniformly integrable.
\end{proof}

\begin{proof}[Proof of Theorem~\ref{thm:asLlogL}]
  We assume without loss of generality that $G_0$ is deterministic. For any $k\geq 0$, let 
  \begin{align*}
      &S_n=\theta_1^{-n-k}\sum_{e\in\mathbb N^{n+k}}w_e f(X_e),\\
      &S'_n=S_n-\mathbb E(S_n\mid\mathcal F_n)=\theta_1^{-n}\sum_{e\in\mathbb N^{n}}w_e X_{e,k}^{f},\\
      &\tilde S_n=\theta_1^{-n}\sum_{e\in\mathbb N^{n}}w_e X_{e,k}^{f}\mathbf 1_{|X_{e,k}^f|\leq \rho^n}
  \end{align*}
   where the $X_{e,k}^{f}$, $e\in\mathbb N^n$, are independent random variables, each with the same distribution  as $X_k^{f}(X_e)$.
 
 \medskip\noindent\textit{Step 1.} We first prove that $S'_n\to0$ almost surely when $n\to+\infty$.
 
  We have
  \begin{align*}
       \frac1\varepsilon\sum_{n=1}^\infty \mathbb E(|S'_n-\tilde S_n|)
      &\leq \frac1\varepsilon\sum_{n=1}^\infty \theta_1^{-n} \mathbb E\left(\sum_{e\in\mathbb N^{n}}w_e X_{e,k}^{f}\mathbf 1_{|X_k^f|> \rho^n}\right)\\
      &= \frac1\varepsilon\sum_{n=1}^\infty \theta_1^{-n} \mathbb E(G_0)Q^ng_{\rho,n,k}^{(1)}<+\infty,
  \end{align*}
  by assumption. We deduce that $S'_n-\tilde S_n$ converges almost surely to $0$ when $n\to+\infty$.
  
  In addition, there exists $c_p>0$ such that
  \begin{align*}
      \sum_{n=1}^\infty   \mathbb E(|\tilde S_n-\mathbb E(\tilde S_n\mid \mathcal F_n)|^p)
      &\leq \sum_{n=1}^\infty c_p \theta_1^{-pn} \mathbb E\left(\sum_{e\in\mathbb N^n} w_e^p|X_{e,k}^f \mathbf 1_{|X_{e,k}^f|\leq \rho^n}-\mathbb E(X_{e,k}^f\mathbf 1_{|X_{e,k}^f|\leq \rho^n}\mid \mathcal F_n)|^p\right)\\
      &\leq 2^{p-1} c_p \sum_{n=1}^\infty \theta_1^{-pn} \mathbb E\left(\sum_{e\in\mathbb N^n} w_e^p|X_{e,k}^f|^p \mathbf 1_{|X_{e,k}^f|\leq \rho^n}\right)\\
      &=2^{p-1} c_p \sum_{n=1}^\infty \theta_1^{-pn} \mathbb E(G_0)(Q^{(p)})^ng_{\rho,k,n}^{(2)}<+\infty
  \end{align*}
  by assumption. We deduce that $\tilde S_n-\mathbb E(\tilde S_n\mid \mathcal F_n)$ converges to $0$ when $n\to+\infty$.
  
  We also observe that
  \begin{align*}
      \sum_{n=1}^\infty \mathbb E\left|\mathbb E(\tilde S_n\mid \mathcal F_n)\right|
      &= \sum_{n=1}^\infty \theta_1^{-n} \mathbb E\left|\mathbb E\left(\sum_{e\in\mathbb N^{n}}w_e X_{e,k}^{f}\mathbf 1_{|X_{e,k}^f|\leq \rho^n}\mid \mathcal F_n\right)\right|\\
      &=\sum_{n=1}^\infty \theta_1^{-n} \mathbb E\left|\mathbb E\left(\sum_{e\in\mathbb N^{n}}w_e X_{e,k}^{f}\mathbf 1_{|X_{e,k}^f|> \rho^n}\mid \mathcal F_n\right)\right|\\
      &\leq \sum_{n=1}^\infty \theta_1^{-n} \mathbb E\left(\sum_{e\in\mathbb N^{n}}w_e |X_{e,k}^{f}|\mathbf 1_{|X_{e,k}^f|> \rho^n}\right)\\
      & = \sum_{n=1}^\infty \theta_1^{-n} \mathbb E(G_0)Q^ng_{\rho,n,k}^{(1)}<\infty,
  \end{align*}
  by assumption. We deduce that $\mathbb E(\tilde S_n\mid \mathcal F_n)$ converges to $0$ almost surely when $n\to+\infty$.
  
  Using the almost sure convergence to $0$ of $S'_n-\tilde S_n$, of $\tilde S_n-\mathbb E(\tilde S_n\mid \mathcal F_n)$, and of $\mathbb E(\tilde S_n\mid \mathcal F_n)$, we deduce that $S'_n$ converges to $0$ almost surely when $n\to+\infty$.

  \medskip\noindent \textit{Step 2.} We now prove the convergence of $u^f_n:=\theta_1^{-n}G_n(f)$. We have, for all $n\geq 1$,
  \begin{align*}
      \mathbb E(\theta_1^{-k}S_n\mid\mathcal F_n)&=\mathbb \theta_1^{-n-k}\sum_{e\in\mathbb N^n}w_e E\left(G^{X_e}_k(f)\right)\\
      &=\mathbb \theta_1^{-n-k}\sum_{e\in\mathbb N^n}w_e E\left(G^{X_e}_{e,k}(f)\mid X_e\right)=\mathbb \theta_1^{-n-k}\sum_{e\in\mathbb N^n}w_e Q^kf(X_e),
      \end{align*}
      where the $G^{X_e}_{e,k}$ are the weighted empirical distributions of the $k$-the generation of independents weighted branching processes with initial position $\delta_{X_e}$.
       For all $k\geq 1$, we have, setting $\varepsilon^f_n:=S_n-\mathbb E(S_n\mid\mathcal F_n)$,
      \begin{align*}
      	\left|u^f_{n+k}-W^f_{n}\right|
      &=\varepsilon^f_n+\left|\theta_1^{-k}\mathbb E(S_n\mid\mathcal F_n)-\theta_1^{-n}\sum_{e\in\mathbb N^n}w_e \eta_{f}(X_e)\right|\\
      &= \varepsilon^f_n+\left|\theta_1^{-n}\sum_{e\in\mathbb N^n}w_e \theta_1^{-k}Q^k(f)(X_e)-\theta_1^{-n}\sum_{e\in\mathbb N^n}w_e \eta_{f}(X_e)\right|\\
      &\leq\varepsilon^f_n+\theta_1^{-n}\sum_{e\in\mathbb N^n}w_e \alpha_k \psi(X_e)\\
      &\leq \varepsilon^f_n+\theta_1^{-n}\sum_{e\in\mathbb N^n}w_e \alpha_k f(X_e)=\varepsilon^f_n+\alpha_k u^f_n,
  \end{align*}
  by assumption EB($f,\psi$).  Since $W_n^f$ converges to $W_\infty^f$ almost surely, since $\varepsilon^f_n$ converges to $0$ almost surely, and since $\alpha_k$ can be chosen arbitrarily small, we deduce that $u^f_n$ converges almost surely to~$W^f_\infty$.

    \medskip\noindent \textit{Step 3.} We conclude the proof of the theorem. Consider a  function $g:\XX\to \mathbb R$ such that EB($g,\psi$) and H($g,k$) hold true for all $k\geq 1$. Then, replacing $f$ by $g$ in steps~1 and~2, we deduce that
    \begin{align*}
    		\left|u^g_{n+k}-W^g_{n}\right|\leq \varepsilon^g_n+\alpha_k\theta_1^{-n}\sum_{e\in\mathbb N^n}w_e  f(X_e)=\varepsilon^g_n+\alpha_k\theta_1^{-n}G_n(f).
    \end{align*}
    Since $\varepsilon^g_n$ goes to $0$ when $n\to+\infty$, since $\alpha_k$ can be chosen arbitrarily small, and since $\theta_1^{-n}G_n(f)$ converges to $W^f_\infty$ (according to the previous step), we deduce that $u^g_n$ converges to~$W^g_\infty$.
\end{proof}

\section{Examples}
    \label{sec:examples}
    
    \subsection{Products of random weights indexed by
        Galton-Watson trees}

    \label{sec:exaGW}
    
    The Mandelbrot cascade referred to in the introduction corresponds to the particular case $\XX=\{1\}$ with initial configuration $G_0=\delta_1$.
       We assume without loss of generality that $\E\left(\sum_{i}u_{i}\right)=1$, so that $(G_n(1))_{n\in\mathbb N}$ is a non-negative martingale and thus it converges almost surely to some non-negative random variable $G_\infty(1)$. The following proposition is an immediate application of Theorem~\ref{thmMain}.
    
\begin{proposition}
Fix $p\in(1,2]$. If
    \begin{align}
        \label{eqNLiu2}
        \E\left[\left(\sum_{i\in\mathbb N}u_{i}\right)^p\right]<+\infty\text{ and
        }\E\left[\sum_{i\in \mathbb N} u_{i}^p\right]<1,
    \end{align}
    where $(u_i)_{i\in\mathbb N}$ is distributed according to $K(1,\cdot)$, then
        \begin{align}
        \sqrt[p]{ \E\left(\left|G_{n}(1)-G_\infty(1)\right|^p\right)}
            \leq&\quad C \,\left[\E\left(\sum_{i\in\mathbb N} u_i^p\right)\right]^{n/p},
    \end{align}
    for some constant $C>0$.
\end{proposition}
    
    In~\cite[Theorem~1]{liu1996products}, assuming that the number $N$ of non-zero weights is finite almost surely and that
    \begin{align}
        \label{eqNLiu1}
       \E\left[\left(\sum_{i\in\mathbb N}u_{i}\right)\log_+\left(\sum_{i\in\mathbb N}u_{i}\right)\right]<+\infty\text{ and }\E\left[\sum_{i\in\mathbb N} u_{i}\log u_{i}\right]<0,
    \end{align}
    the author shows that~\eqref{eqNLiu2} is actually equivalent to  $\E((G_\infty(1)^p)<+\infty$.  Our result thus precise this equivalence by showing that, under these assumptions, $\E((G_\infty(1)^p)<+\infty$ is actually equivalent to the exponential convergence of $G_n(1)$ to $G_\infty(1)$.
    
Applying Theorem~\ref{thm:unifBiggins}, we deduce the following result.
\begin{proposition}
If there exists $p\in(1,2]$ such that
    \begin{align}
        \label{eqNLiu3}
        \E\left[\left(\sum_{i\in\mathbb N}u_{i}\right)\log_+\left(\sum_{i\in\mathbb N}u_{i}\right)\right]<+\infty\text{ and
        }\E\left[\sum_{i\in \mathbb N} u_{i}^p\right]<1,
    \end{align}
    then the martingale $(G_n(1))_{n\in\mathbb N}$ is uniformly integrable.
\end{proposition}

In \cite[Theorème~1.1]{Liu1997}, assuming that the number $N$ of non-zero weights is finite almost surely, the author shows that the conclusion is equivalent to the weaker condition~\eqref{eqNLiu1}. This shows that our $L\log L$ results can be optimized at least under the restrictions considered in~\cite{Liu1997}. Since the point of view of~\cite{Liu1997} is quite different, it may be possible to combine both strategies to obtain sharper $L\log L$ results. This is the subject of an ongoing work.

\subsection{Products of random kernels indexed by Galton-Watson trees}

\label{sec:exakernels}

We consider a generalization of the previous section's model by considering the situation where to each individual $e$ in a Galton-Watson tree is attached a random kernel $A_e$, and we are interested in the asymptotic behaviour of the sum over the individuals of generation $n$ of the products of the kernels along their ancestry line, that is
$\sum_{e\in\mathbb N^n} \prod_{0<e'\leq e} A_{e'}$, where $e'\leq e$ means that $e'$ is an ancestor of~$e$.

More formally, we consider the situation where $\XX$ is a measurable set of bounded kernels on a measurable space $E$, where $G_0=\delta_I$, with $I:E\to E$ is the identity kernel, and where there exists a probability measure $\mu_\XX$ on $\XX^\mathbb N$ such that, for all $A\in \XX$,
\[
K(A,\cdot)\text{ is the law of }(1,AA_i)\text{ where }(A_i)_{i\in\mathbb N}\sim \mu_\XX.
\]
Note that we take all the weights equal to $1$ in this case, since the weights can be encoded in the random kernels. 
Denoting by $A_e\in \XX$ the random kernel attached to the individual $e\in \mathbb N^{\mathbb N}$, we have 
$G_n(F_{x,f})=\sum_{e\in\mathbb N} \prod_{0<e'\leq e} \delta_x A_{e'} f$.

We assume that, for all $x\in E$,
\begin{align*}
    \mathbb E\left(\sum_{i\in\mathbb N} \delta_x|A_i|(E) \right)<\infty,
\end{align*}
where $(A_i)_{i\in\mathbb N}$ is distributed according to $\mu_\XX$ and $|A_i|=(A_i)_++(A_i)_-$, with  $(A_i)_+$ (resp. $(A_i)_-$) the positive (resp. negative) part of the kernel $A_i$. We define the (determinisitic) kernel $P$ on $E$ by
\begin{align*}
    \delta_x Pf:=\mathbb E_x\left(\sum_{i\in \mathbb N} \delta_x A_i f\right),
\end{align*}
for all $x\in E$ and bounded measurable $f:E\to\mathbb R_+$.

In the following result, we use the kernel norm $||| A|||=\sup_{x\in E}\delta_x|A|(E)$
\begin{proposition}
Fix $p\in(1,2]$. Assume that
    there exists $\theta_1>0$, an integer $\beta\geq 0$, a bounded function $\eta:E\to\mathbb R$ and a totally bounded measure $\nu$ on $E$ such that
    \begin{align}
        \label{eq:Pninprop}
        \left|n^{-\beta}\theta_1^{-n}\delta_x P^n f-\eta(x)\nu(f)\right|\leq \alpha_n\|f\|_\infty,\quad\forall x\in E \text{ and  bounded measurable }f:E\to\mathbb R,
    \end{align}
    where $\alpha_n$ goes to $0$ when $n\to+\infty$. We also assume that
    \begin{align}
        \label{eqNLiu2bis}
        \E\left[\left(\sum_{i\in\mathbb N}|||A_i|||\right)^p\right]<+\infty\text{ and
        }\bar\gamma:=\E\left[\sum_{i\in \mathbb N} |||A_i|||^p\right]<\theta_1^p,
    \end{align}
    where $(A_i)_{i\in\mathbb N}$ is distributed according to $\mu_\XX$.
    Then, for all bounded measurable function $f:E\to\mathbb R_+$ and all $x\in E$,
        \begin{multline*}
        \sqrt[p]{ \E\left(\left|n^{-\beta}\theta_1^{-n}\mathbb E\left(\sum_{e\in\mathbb N^{n+m}} \delta_x \prod_{e'\leq e}A_{e'}\,f\right)-W^\eta_\infty\nu(f)\right|^p\right)}\\
        \leq C
\|f\|_{\infty}\,\left(\left(\frac{\bar \gamma}{\theta_1^1}\right)^{m/p}m^{(1/2-1/p)_+}  +\alpha_n \dfrac{n^\beta m^\beta}{(n+m)^\beta} +  \left(1-\dfrac{n^\beta}{(n+m)^\beta}\right)\right),
    \end{multline*}
    for some constant $C>0$.
\end{proposition}

\begin{remark}
    In the previous result, we work in the space of bounded measures/functions. Our result also applies to kernels bounded in weighted spaces of type $L^\infty(V)$, using $V$ transforms (see~\cite{ChampagnatVillemonais2019} for details in the context of quasi-stationary distributions). Indeed, the $V$-transform of a WBP is still a WBP, while this is not the case, but in particular instances, for usual branching processes. 
\end{remark}

\begin{proof}
For all $x\in E$, all bounded $f:E\to\mathbb R$ and all kernel $A\in \XX$, we let $F_{x,f}(A)=\delta_x A f$. We observe that
\begin{align*}
    QF_{x,f}(A)&=\mathbb E\left(\sum_{i\in \mathbb N} \delta_x AA_i f\right)=\delta_x A\mathbb E\left(\sum_{i\in \mathbb N} A_i f\right)=\delta_x APf=F_{x,Pf}(A).
\end{align*}
By iteration, we deduce that, for all $n\geq 0$,
\begin{align*}
   \delta_AQ^n F_{x,f}=F_{x,P^n f}(A)=\delta_x AP^n f.
\end{align*}
This shows that the first part of condition MD($F_{x,f},p$) holds true for $\bar K$ with the same $\beta$, $\theta_1$ and $\alpha_n\|f\|_\infty$ (instead of $\alpha_n$), with $\psi_1(A)=|||A|||$ and $\eta_{F_{x,f}}(A)=(\delta_xA\eta)\nu(f)$.

We now check that second part of MD($F_f,p$) holds true for $\bar K$ with $\psi_2(A)=|||A|||$. We have
\begin{align*}
    \delta_A Q(\psi_2^p)&=\mathbb E\left(\sum_{i\in\mathbb N} |||A A_i|||^p\right)\\
    &\leq |||A|||^p\,\mathbb E\left(\sum_{i\in\mathbb N} |||A_i|||^p\right)\\
    &= \psi_2(A)^p\bar \gamma.
\end{align*}
Since $Q^{(p)}=Q$, we deduce that, for all $n\geq 1$,
 \begin{align*}
            \theta_1^{-pn}\delta_A (Q^{(p)})^n (\psi_2^p)
            &= 
            \theta_1^{-pn}\bar \gamma^n\psi_2(A)^p\\
            &= \theta_1^{-pn}\,\bar\gamma^n n^{(p/2-1)_+}\,\psi_2(A)^p\,/\,n^{(p/2-1)_+},
\end{align*}
and, for all measurable function $g:\XX\to \mathbb R$ and all $x\in \XX$,
\begin{align*}
\E\left(\left|\sum_{i\in\mathbb N} g(AA_i)-\delta_x Q g\right|^p\right)\leq 2^p\,\E\left(\left|\sum_{i\in\mathbb N} |||A A_i||| \right|^p\right)\|g\|_{\psi_1}^p\leq C'\psi_2(A)^p\|g\|_{\psi_1}^p,
\end{align*}
for some constant $C'>0$ and where we used the assumption that $\E\left(\left(\sum_{i\in\mathbb N} |||A_i||| \right)^p\right)<\infty$.

Hence Theorem~\ref{thmMain} applies. Using the fact that, for any $a\in(0,1)$, there exists some constant $C'>0$ such that
\[
\sum_{k\geq m} a^n n^q\leq C' a^m m^q,
\]
this concludes the proof.
\end{proof}

\subsection{Random dynamics with contraction in Wasserstein distance}

\label{sec:exaWasserstein}

 Assume  that $\XX$ is endowed with a bounded metric $d$ and is a Polish space. Let $M$ be a measurable space, and consider a family $(f_{\zeta})_{\zeta \in M}$ of Lipschitz functions from $\XX$ to $\XX$ such that, for all $\zeta \in M$, $l_{\zeta} := \|f_\zeta\|_{\text{Lip}(d)} < 1$. Let $K_u$ be a probability kernel  from $\XX$ to $(\mathbb R_+)^\mathbb N$ and $\mu_\zeta$ be a probability measure  on $M^\mathbb N$, such that the marginal laws of $\mu_\zeta$ are identical.
We consider the weighted branching process with reproduction kernel $K$ defined from $\XX$ to $(\mathbb R_+\times \XX)^\mathbb N$ by
\[
K(x,\cdot)\text{ is the law of }(u_i,f_{\zeta_i}(x))_{i\in\mathbb N},\ \text{where $((u_i)_{i\in\mathbb N},(\zeta_i)_{i\in\mathbb N})$ is distributed according to $K_u(x,\cdot)\otimes\mu_\zeta$.}
\]
Informally, given an individual with type $x\in \XX$, its progeny is given by individuals with weights $u_i$ and types $f_{\zeta_i}(x)$, $i\in \mathbb N$.

In the following result, we let
\begin{align*}
J:x\in \XX\mapsto \E\left(\sum_{i\in\mathbb N} u_i\log_+ u_i\right).
\end{align*}
and , for any $q\geq 1$,
\begin{align}
    H_q:x\in \XX\mapsto \E\left(\left|\sum_{i\in\mathbb N} u_i \right|^q\right)\text{ and }L_q:x\in \XX\mapsto \E\left(\sum_{i\in\mathbb N} u_i^q\right).
\end{align}
where $(u_i)_{i\in \mathbb N}$ is distributed according to $K_u(x,\cdot)$.

\begin{proposition}
\label{prop:Wasserstein}
Fix $p\in(1,2]$. Assume that $H_p$ is bounded, and that $H_1$ is Lipschitz and bounded away from $0$ by a positive constant. Then, for any Lipschitz function $f:\XX\to\mathbb R$, the first part of Assumption~MD($f,p$) holds true with $\psi_1\equiv 1$, $\beta=0$, a probability measure $\nu_1$ on $\XX$, and $\eta_f(x)=\eta(x) \nu_1(f)$ for some positive Lipschitz function $\eta$, some $\theta_1>0$, and $\alpha_n=c\,a^n\,\|f\|_{Lip}$ for some constants $c>0$ and $a\in (0,1)$.
    
    In addition, 
    \begin{enumerate}
        \item if there exists $\bar\gamma\in(0,1)$ such that
   $
    L_p(x)\leq \bar\gamma\theta_1^{p-1} L_1(x)
    $ for all $x\in \XX$,
        then there exists a constant $C>0$ such that
    \begin{multline}
    \label{eq:eqpropWar}
        \sqrt[p]{ \E\left(\left|\theta_1^{-n-m}G_{n+m}(f)-W^\eta_\infty\nu_1(f)\right|^p\right)}
            \leq C \left(\|f\|_{\infty}+\|\eta\|_{\infty}\nu_1(|f|)\right)\,\bar\gamma^{m/p}m^{(1/2-1/p)_+}  \sqrt[p]{\E \left(G_0^{(p)} (\XX)\right)}\\
            +C a^n\|f\|_{Lip}  \left( \sqrt[p]{\E \left(G_0^{(p)} (\XX)\right)} 
            + \sqrt[p]{\E\left(\left|G_0(\XX)\right|^p\right)}\right).
    \end{multline}
        where
        $W_\infty^\eta$ is the almost sure and $L^p$ limit of
        the uniformly
        integrable martingale $W^\eta:=(\theta_1^{-n}G_n(\eta))_{n\in\N}$. 
        In addition, $\theta_1^{-n}G_{n}(f)$ converges almost surely to $W^\eta_\infty\nu_1(f)$ when $n\to+\infty$. 
        \item if $\nu_1(J)<\infty$ and $\mathbb E(G_0)\leq C \nu_1$ for some constant $C>0$, then $W^\eta$ is uniformly integrable and $\theta_1^{-n}G_{n}(f)$ converges almost surely to $\nu_1(f)W^\eta_\infty$. 
    \end{enumerate}
\end{proposition}

\begin{remark}
    We make use of~\cite{champagnat2025uniform}, where the authors prove a general convergence criterion in Wasserstein distance for non-conservative semi-groups. The proof of Proposition~\ref{prop:Wasserstein} can be extended to examples that enter the general setting therein.
\end{remark}

\begin{proof}
We prove the first part of the proposition in Step~1, and the two points of the second part in Step~2.

\medskip\noindent\textit{Step 1.}
  Let 
  \[
  p:x\in \XX\mapsto \frac{\delta_x Q(\XX)}{\sup_y \delta_y Q(\XX)},
  \]
  which is Lipschitz by assumption. We consider the Markov chain $(X_n)_{n\in\mathbb N}$ on $\XX$ with transition probabilities given by
  \[
  X_1=f_\zeta(x),\quad \text{with}\quad \zeta\text{ distributed according to a marginal law of $\mu_\zeta$.}
  \]
  Note that, because of the independence of $(u_i)_{i\in\mathbb N}$ and $(\zeta_i)_{i\in\mathbb N}$, and the fact that the $\zeta_i$ have all the same law, we have $\forall A\in \mathcal X$ and all $x\in \XX$,
  \begin{align*}
  \mathbb P(X_1\in A\mid X_0=x) &=\frac{\mathbb E\left(\sum_{i\in\mathbb N}u_i\right)\mathbb E\left(\mathbf 1_{f_{\zeta}(x)\in A}\right)}{\mathbb E\left(\sum_{i\in\mathbb N}u_i\right)}\\
  &=\frac{\mathbb E\left(\sum_{i\in\mathbb N}u_i\mathbf 1_{f_{\zeta_i}(x)\in A}\right)}{\mathbb E\left(\sum_{i\in\mathbb N}u_i\right)}\\
  &=\frac{\delta_x Q(A)}{\delta_x Q(\XX)}. 
  \end{align*}
  In particular, 
  \begin{align}
      \mathbb E_x\left(p(x)p(X_1)\cdots p(X_{n-1})\,\mathbf 1_{X_n\in A}\right)&=\mathbb E_x\left(p(x)p(X_1)\cdots p(X_{n-2})\frac{\delta_{X_{n-1}}Q(A)}{\sup_y \delta_y Q(\XX)}\right)\nonumber\\
      &=\cdots=\frac{\delta_x Q^n(A)}{(\sup_y \delta_y Q(\XX))^n}
      \label{eq:kernelkilledchain}
  \end{align}

  Since $p$ is Lipschitz, \cite[Theorem~2.4]{CSV2023} entails that, for some $C_1>0,a\in(0,1)$, for all probability measures $\mu,\nu$ on $\XX$,
  \begin{align*}
      \mathcal W_d(\mathbb Q_\mu(X_n\in \cdot),\mathbb Q_\nu(X_n\in \cdot))\leq C_1a^n\mathcal W_d(\mu,\nu),
  \end{align*}
  where, for some $\theta_0>0$ and some positive Lipschitz function $\eta:\XX\to(0,+\infty)$,
  \begin{align*}
      \mathbb Q_x\left(X_n\in\cdot\right)&=\frac{\theta_0^{-n}}{\eta(x)}\mathbb E_{x}\left(\eta(X_n)\mathbf 1_{X_n\in \cdot}p(x)p(X_1)\ldots p(X_{n-1})\right)\\
      &= \frac{\delta_x Q^n(\eta\mathbf 1_{\cdot})}{\eta(x)\theta_0^n(\sup_y \delta_y Q(\XX))^n},
  \end{align*}
  and where $\mathcal W_d$ is the Wasserstein metric associated to $d$. In addition, there is a probability measure $\nu_{Q}$ on $\XX$ such that $\mathbb Q_{\nu_Q}(X_n\in\cdot)=\nu_Q$ for all $n\geq 0$.
  
  Setting $\theta_1=\theta_0\sup_y \delta_y Q(\XX)$, we deduce that, for all Lipschitz function $g:\XX\to\mathbb R$,
  \begin{align*}
      \left|\theta_1^{-n}\delta_x Q^n(\eta g)-\eta(x)\nu_Q(g)\right|\leq \eta(x)C_1 a^n \|g\|_{Lip}.
  \end{align*}
  We deduce that, for all Lipschitz function $f$ on $\XX$ and setting $g=f/\eta$ in the previous expression, we have
   \begin{align*}
      \left|\theta_1^{-n}\delta_x Q^n(f)-\eta(x)\nu_Q(f/\eta)\right|\leq \|\eta\|_{\infty}C_1 a^n \|f/\eta\|_{Lip}.
  \end{align*}
  Since $\eta$ is Lipschitz bounded away from $0$ and $+\infty$, we deduce that there exists a constant $C_2$ such that $\|f/\eta\|_{Lip}\leq C_2 \|f\|_{Lip}$ for any Lipschitz function $f$.
 Defining the probability measure $\nu_1=\nu_Q(\cdot/\eta)$ (see~\cite{CSV2023} to see why $\nu_1$ is a probability measure), we deduce that the first part of the proposition holds true.
  
  \medskip\noindent \textit{Step 2.} To prove point 1., it remains to check the second part of MD($f,p$). It holds clearly with $\psi_2\equiv 1$ and $\gamma_n^p=\bar\gamma^n n^{(p/2-1)_+}$. The series with general term $\gamma_n$ is summable with tail $\Gamma_n\sim C_3\,\bar\gamma^{n/p}n^{(1/2-1/p)_+}$ for some constant $C_3>0$. Applying Theorem~\ref{thmMain}, this implies the the first point of the second part of the proposition. Point 2. is an immediate consequence of Theorems~\ref{thm:unifBiggins} and~\ref{thm:asLlogL}.
  \end{proof}

\subsection{Ergodic averages along lineages}
\label{sec:exainhomogeneous}

We assume that we are given a kernel $K$ as in the introduction, and the associated WBP $(G_n)_{n\in\mathbb N}$. We are interested in the long time behaviour of the ergodic averages along lineages, that is in
\[
A_n:=\sum_{e\in \mathbb N^n} w_e M_e, \quad \text{ where }M_e:=\frac{1}{|e|}\sum_{e'\leq e} \delta_{X_{e'}},
\]
where $|e|=n$ is the length of $e$ and $e'\leq e$ means that $|e'|\leq |e|$ and $e'_k=e_k$ for all $k\in\{1,\ldots,|e'|\}$. 

We assume that there exist $\theta_1>0$, a bounded function $\eta:\XX\to\mathbb R_+$, a probability measure $\nu$ on $\XX$ such that,  for all bounded measurable function $f:\XX\to \mathbb R_+$,
\begin{align}
\label{eq:assumption1}
    \left|\theta_1^{-n}\delta_x Q^n f-\eta(x)\nu(f)\right|\leq \alpha_n\|f\|_{\infty},
\end{align}
where $(\alpha_n)_n$ is a summable family of positive numbers.

Recall that $J$, $H_1$, $H_p$, $L_1$ and $L_p$ are defined in the previous section.

\begin{proposition}
Fix $p\in(1,2]$. Assume that~\eqref{eq:assumption1} holds true and that $H_1$ and  $H_p$ are bounded. Assume in addition that there exists $\bar\gamma\in(0,1)$ such that
   $
    L_p(x)\leq \bar\gamma\theta_1^{p-1} L_1(x)
    $ for all $x\in \XX$.
        Then there exists a constant $C>0$ such that, for all measurable function $f$,
    \begin{align*}
        \theta_1^{-n}A_{n}(f)\xrightarrow[n\to+\infty]{L^p} W^\eta_\infty\nu(\eta f),
    \end{align*}
        where
        $W_\infty^\eta$ is the almost sure and $L^p$ limit of
        the uniformly
        integrable martingale $W^\eta:=(\theta_1^{-n}G_n(\eta))_{n\in\N}$. 
\end{proposition}

\begin{proof}
We build a new WBP on the type space $\bar \XX=\XX^{(\mathbb N)}$ with kernel $\bar K$ from $\bar \XX$ to $(\mathbb R_+\times \bar \XX)^{\mathbb N}$ defined, for all $\bar x=(x_0,\ldots,x_n)\in \XX^n$ for some $n\geq 0$, 
\[
\bar K(\bar x,\cdot)\text{ is the law of }(u_i,\bar xy_i), \text{ where $(u_i,y_i)_{i\in\mathbb N}$ has law $K(x_n,\cdot)$,}
\]
and $\bar x y_i=(x_0,\ldots,x_n,y_i)$. We denote by $\bar Q$ the associated kernel from $\bar \XX$ to $\bar \XX$.

Let $f:\XX\to \mathbb R$ bounded an measurable. We define, for all $\bar x=(x_0,\ldots,x_n)\in \XX^n$, the function
\[
\Sigma(\bar x)=\sum_{i=0}^n f(x_i)
\]
We have
\begin{align*}
\delta_{\bar x}\bar Q\Sigma&=\mathbb E\left(\sum_{i\in\mathbb N}u_i\Sigma(\bar xy_i)\right)=\mathbb E\left(\sum_{i\in \mathbb N}u_i (\Sigma(\bar x)+f(y_i))\right)=
\delta_{x_n}(Q\mathbf 1_\XX)\,\Sigma(\bar x)+\delta_{x_n}Qf,
\end{align*}
and, by iteration, for all $m\geq 1$,
\[
\delta_{\bar x}\bar Q^m\Sigma=\delta_{x_n}Q^m(\mathbf 1_\XX)\Sigma(\bar x)+\delta_{x_n}Q(Q^{m-1}(\mathbf 1_\XX)f)+\cdots+\delta_{x_n}Q^k(Q^{m-k}(\mathbf 1_\XX)f)+\cdots+\delta_{x_n}Q^m(f).
\]
We deduce that, for some constant $C'>0$ that may change from line to line,
\begin{align*}
    \left|\theta_1^{-m}\delta_{\bar x}\bar Q^m\Sigma-(m+1)\,\eta(x_n)\nu(\eta f)\right|
    &\leq \theta_1^{-m}\delta_{x_n}Q^m(\mathbf 1_\XX)\Sigma(\bar x)+\sum_{k=1}^m \left|\theta_1^{-m}\delta_{x_n}Q^k(Q^{m-k}(\mathbf 1_\XX)f)-\eta(x_n)\nu(\eta f)\right|\\
    &\leq C'n\|f\|_\infty +\sum_{k=1}^m \left|\theta_1^{-k}\delta_{x_n}Q^k((\theta_1^{-(m-k)}Q^{m-k}(\mathbf 1_\XX)-\eta)f)\right|\\
    &\qquad+\sum_{k=1}^m\left|\theta_1^{-k}\delta_{x_n}Q^k(\eta f)-\nu(\eta f)\right|,
\end{align*}
where we used the fact that $\eta$ is a right eigenfunction for $Q$. Then, using~\eqref{eq:assumption1},
\begin{align*}
    \left|\theta_1^{-m}\delta_{\bar x}\bar Q^m\Sigma-(m+1)\,\eta(x_n)\nu(\eta f)\right|
    &\leq C'n\|f\|_\infty +C'\sum_{k=1}^m\alpha_{m-k}\|f\|_{\infty}+C'\sum_{k=1}^m\alpha_k\|f\|_{\infty}.\\
    &\leq C'(n+2)\|f\|_\infty.
\end{align*}
We deduce that for $F:\bar x\to\frac{1}{|\bar x|}\Sigma(\bar x)$,
\begin{align*}
    \left|\theta_1^{-m}\delta_{\bar x}\bar Q^mF-\eta(x_n)\nu(\eta f)\right|
    &\leq C'\frac{n+2}{m}\|f\|_\infty.
\end{align*}
Hence the WBP build on $\bar \XX$ satifies the assumptions of Theorem~\ref{thmMain} with $\bar \alpha$ depending on the size of $\bar x$. This allows us to conclude the proof (see Remark~\ref{rem:alphamn}).
\end{proof}

\end{document}